\documentclass[11pt,dvipsnames,reqno,hypertexnames=false]{amsart}
\usepackage[T1]{fontenc}
\usepackage{a4wide}
\usepackage{yhmath,mathrsfs,amsthm,amsmath,amssymb,amsfonts,enumerate,lipsum, appendix,mathtools, bm}
\usepackage{bbold}
\usepackage{pdfpages}
\usepackage{orcidlink}
\usepackage[mathscr]{euscript}
\usepackage{graphicx}
\usepackage{pgf,tikz}
\usepackage{tkz-euclide}
\usepackage{graphicx}
\usepackage{subcaption}
\usetikzlibrary{shapes.geometric}
\usetikzlibrary{arrows,hobby}
\usepackage[shortlabels]{enumitem}
\usepackage[english]{babel}

\allowdisplaybreaks

\usepackage[sort,numbers]{natbib}
\usepackage{thmtools}
\usepackage{hyperref}
\usepackage{esint}
\hypersetup{
	colorlinks=true,
	linkcolor=blue,
	filecolor=blue,      
	urlcolor=green,
	citecolor=green,
}
\definecolor{green}{rgb}{0.22, 0.88, 0.08}
\usepackage[margin=0.90in, heightrounded]{geometry}
\usepackage{bigints}
\newtheorem{theorem}{Theorem}[section]
\newtheorem{lemma}[theorem]{Lemma}
\newtheorem{proposition}[theorem]{Proposition}

\newtheorem{definition}[theorem]{Definition}
\theoremstyle{definition}

\newtheorem{remark}[theorem]{Remark}
\numberwithin{equation}{section}
\usepackage[hyperpageref]{backref}
\DeclarePairedDelimiter\abs{\lvert}{\rvert}
\DeclarePairedDelimiter\norm{\lVert}{\rVert}
\makeatletter
\let\oldnorm\norm
\def\norm{\@ifstar{\oldnorm}{\oldnorm*}}

\DeclareMathOperator*{\esssup}{ess\,sup}
\DeclareMathOperator*{\essinf}{ess\,inf}

\makeatletter
\newcommand{\al} {\alpha}

\newcommand{\be} {\beta}

\newcommand{\De} {\Delta}

\newcommand{\om} {\omega}
\newcommand{\Om} {\Omega}
\newcommand{\la} {\lambda}

\newcommand{\sig} {\sigma}

\newcommand{\Gr} {\nabla}
\newcommand{\no} {\nonumber}
\newcommand{\noi} {\noindent}
\newcommand{\var} {\varepsilon}
\newcommand{\ra} {\rightarrow}

\DeclareMathAlphabet{\mathpzc}{T1}{pzc}{m}{it}

\def\w{{\widetilde w}}

\def\Tail{\text{\rm Tail}}

\def\dx{{\,\rm d}x}
\def\dy{{\,\rm d}y}
\def\dxy{{\,\rm d}x{\,\rm d}y}

\def\dm{\,{\rm d}\mu}

\def\sb2{{{\mathcal D}^{1,2}_0(B_1^c)}}

\def\w2r{{{ W}^{2,2}(\R^N)}}

\def\d2{{{\mathcal D}^{2,2}_0(\Om)}}

\newcommand*\rd{\mathbb{R}^d}
\def\A{{\mathcal A}}
\def\C{{\mathcal C}}

\def\R{{\mathbb R}}

\def\F{{\mathcal F}}

\def\({{\Big(}}
\def\){{\Big)}}

\def\ws2{{\F_{\frac{N}{2}}}}
\def\L2{{ L^{1,\;\infty}(\log L)^2}}
\def\l2{\mathcal M\log L}

\def\c1Loc{{\C_{loc}^1}}

\def\Gag2{\iint_{\R^{2d}} \frac{(u(x)-u(y))^2}{|x-y|^{N+sp}}\ \dxy}

\def\Gagn2{\iint_{\R^{2d}} \frac{(u_n(x)-u_n(y))^2}{|x-y|^{N+sp}}\ \dxy}
\DeclareMathOperator\supp{supp}
\usepackage{natbib}

\usepackage{xpatch}
\makeatletter   
\xpatchcmd{\@tocline}
{\hfil\hbox to\@pnumwidth{\@tocpagenum{#7}}\par}
{\ifnum#1<0\hfill\else\dotfill\fi\hbox to\@pnumwidth{\@tocpagenum{#7}}\par}
{}{}
\def\l@subsection{\@tocline{2}{0pt}{2pc}{6pc}{}}
\def\l@subsubsection{\@tocline{3}{0pt}{8pc}{8pc}{}}
\makeatother
\vspace{-0.5\baselineskip}
\usepackage{orcidlink}
\usepackage{fancyhdr}
\title[Harnack inequality for mixed local and nonlocal $p$-Laplace equations]{Harnack Inequality for Mixed Local-Nonlocal weighted homogeneous equations}
\author[N. Biswas and S. Das]{Nirjan Biswas$^1$\,\orcidlink{0000-0002-3528-8388} and Stuti Das$^2$\,\orcidlink{0009-0007-9968-1510}}
\address{\rm$^1$\,Department of Mathematics, Indian Institute of Science Education and Research Pune, Dr. Homi Bhabha Road, Pune 411008, India}
\address{\rm$^2$\,Department of Mathematics and Statistics, Indian Institute of Technology Kanpur, Kanpur-208016, India}
\email[N. Biswas]{nirjan.biswas@acads.iiserpune.ac.in, nirjaniitm@gmail.com} 
\email[S. Das]{stutid21@iitk.ac.in}
\thanks{$^1$Corresponding author}	

\begin{document}
\begin{abstract}
We consider the following class of mixed local-nonlocal equations:
\begin{align}\label{abs}\tag{$\mathcal{P}$}
    -\De_p u + (-\De)_p^s u = V \abs{u}^{p-2}u \text{ in } \Omega,
\end{align} 
where $s \in (0,1), p \in (1, \infty)$, and the weight function $V$ lies in scaling subcritical Lebesgue space $L^q(\Omega)$ where $q>\frac{d}{p}$ when $d>p$ and $q>1$ when $d \le p$. We establish Harnack inequality for weak solution and weak Harnack inequality for weak supersolution to \eqref{abs}. Our approach is based on the De Giorgi-Nash-Moser theory, the expansion of positivity and estimates involving a tail term. Our results also apply to integro-differential operators, with the prototype given by $(-\De)_p^s$. This work generalizes some regularity results of Garain-Kinnunen (Trans. Am. Math. Soc., 375(8), 2022) and Garain (Nonlinear Anal., 256, 2025) to the setting of general weight functions.
\end{abstract}
\maketitle

\section{Introduction}
This paper studies Harnack inequality for a jump-diffusion process, which is described by the following homogeneous equation: 
\begin{align}\label{main_PDE}\tag{$\mathcal{P}$}
    -\De_p u + (-\De_p)^s u = V \abs{u}^{p-2}u \text{ in } \Omega,
\end{align}
where $p \in (1, \infty)$ is the integrability exponent, $s\in (0,1)$ is the differentiability parameter, $\Omega$ is a bounded open set in $\rd$ and the weight function $V$ satisfies
\begin{equation}\label{weight}
V \in L^q(\Omega) \text{ with } \left\{\begin{aligned}
&q>\frac{d}{p},\,&\text{if}\;d>p; \\
&q> 1,\,&\text{if}\;d\le p.
\end{aligned}
\right.
\end{equation}
The $p$-Laplace operator $\De_p$ and the fractional $p$-Laplace operator $(-\De)_p^s$ are defined as
\begin{align*}
    \De_p u = \text{div}(\abs{ \Gr u}^{p-2} \Gr u), \text{ and } (-\De_p)^s u = \text{P.V.} \int _{\rd} \frac{|u(x)-u(y)|^{p-2}(u(x)-u(y))}{|x-y|^{d+ps}} \dy, \, x \in \rd,
\end{align*}
respectively, where P.V. means “in the principal value sense”.

The classical Harnack inequality, formulated in 1887, asserts that for any nonnegative harmonic function $u: B_1 \rightarrow \R$, there exists a constant $C>0$ such that the inequality $u(x) \leq C u(y)$ holds for every $x, y \in B_{1/2}$ (where $B_r \subset \rd$ is a ball of radius $r$ with centre at origin). This inequality is known as the Harnack inequality, and it received significant interest after Moser in \cite{Mo1961} demonstrated that Harnack inequality leads to a priori estimates in H\"{o}lder spaces. Subsequently, De-Giorgi (1956), Nash (1958), and Moser (1964) independently established the Harnack inequality for the weak solutions to $- \text{div} (A(x) \nabla u) = 0$ in $B_1$, where $A$ is a bounded, measurable, and positive definite function. More precisely, for a nonnegative weak solution $u$, there exists a constant $C>0$ such that
\begin{align*}
    \sup_{B_{\frac{1}{2}}} u \le C \inf_{B_{\frac{1}{2}}} u. 
\end{align*}
This result has significant value in solving Hilbert's 19th Problem. Since then, Harnack inequalities for weak solutions of various local elliptic operators have been extensively studied. We refer to \cite{MalyZiemer, HL1997, Evans} for detailed descriptions of this study. In \cite{CFG1986}, Chiarenza-Fabes-Garofalo established the Harnack inequality for any local positive weak solution to $-\Delta u = fu \; \text{in} \; \Omega,$ where the function $f$ lies in a scaling critical class, namely in the Stummel class of potentials. Later, in \cite[Theorem 2.2]{B2001}, Biroli extended this result to a nonlinear set-up. It is shown that if $u$ is a local positive weak solution to $-\De_p u = f |u|^{p-2}u$ in $\Omega$, where $f$ lies in the Kato space, then there exists $C>0$ such that 
\begin{align*}
    \sup_{B_r(x)} u \le C \inf_{B_r(x)} u,
\end{align*}
for every $B_r(x) \subsetneq \Omega$. Their proof mainly followed Moser's approach with the application of classical John-Nirenberg Lemma (see \cite{John-Nirenberg} and \cite[Theorem 3.5]{HL1997}).
  
We now highlight the contribution of Castro-Kuusi-Palatucci \cite{KuusiHarnack}, who initiated the study of Harnack inequalities for a quasilinear nonlocal operator. They introduced the \textit{tail term}, denoted as $\Tail_{p-1,sp,p}$ (see \eqref{mtail}) in proving the Harnack estimates for the following problem 
\begin{align}\label{Pala}
    (-\De_p)^s u = 0 \text{ in } \Omega, \; u=g \text{ in } \rd \setminus \Omega,
\end{align}
where $\Omega$ is a bounded open set in $\mathbb{R}^d$. More precisely, they have shown that if $u$ is a local nonnegative weak solution to \eqref{Pala}, then there exists a constant $C=C(d,p,s)>0$ such that
\begin{align}\label{nonlocal-harnack}
    \sup_{B_r(x)} u \le C \inf_{B_r(x)} u + C\left( \frac{r}{R}\right)^{\frac{sp}{p-1}} \Tail_{p-1,sp,p}(u^-,x; R),
\end{align}
for every $B_r(x) \subset B_{R/2}(x)$ with $x \in \Omega$ and $B_R(x) \subset \Omega$. In particular, if $u$ is nonnegative in $\rd$, then \eqref{nonlocal-harnack} reduces to the classical Harnack inequality. A counterexample due to Kassmann \cite{Kas2011} shows that such positivity assumptions cannot be removed or weakened to get classical Harnack inequality, even in the case $p=2$, i.e., for the fractional Laplacian $(-\Delta)^s$. In \cite{KuusiHarnack}, the authors also studied a weak Harnack type inequality for nonnegative weak supersolutions to \eqref{Pala}. It is worth noting that the authors have employed De Giorgi's approach to establish these regularity results. We also mention the works of \cite{Br2018, Anup2025, Garain2024, Palatucci2016, giova2025}, where significant developments concerning interior Hölder regularity, higher Hölder regularity, Lipschitz regularity, and $\mathcal{C}^{1,\alpha}$ regularity for weak solutions to $(-\Delta_p)^s u =f$ in $\Omega$, (for a suitable potential function $f$) have been investigated.  

To provide proper context for our results, we now review the literature associated with the regularity of mixed local-nonlocal problems. We begin with mentioning the work of Garain-Kinnunen \cite{garainjuha}, where, following the ideas developed in \cite{KuusiHarnack}, the authors have obtained the local boundedness, Harnack inequality, weak Harnack inequality and local H\"older regularity for weak solutions to
\begin{equation}\label{mixed}
- \Delta_p u+(-\Delta_p)^s u=0 \text{ in } \Omega,
\end{equation}
where $\Omega$ is a bounded open set in $\mathbb{R}^d$.
The higher Hölder regularity and almost Lipschitz regularity for weak solutions to \eqref{mixed} are
obtained by Garain-Lindgren \cite{garain2023higher}. In \cite{MG}, Filippis-Mingione studied the following mixed local-nonlocal problem with a potential $f$:
\begin{equation}\label{MinFi}
- \Delta_p u+(-\Delta_r)^s u=f \text{ in }  \Omega,
\end{equation}
where $p,r\in (1, \infty)$ with $p\geq sr$ and $\Omega$ is a bounded open set in $\mathbb{R}^d$. It is shown that weak solutions to \eqref{MinFi} are locally Hölder continuous for every $\alpha\in(0,1)$ when $f \in L^d(\Omega)$, and their gradients are locally Hölder continuous for some exponent $\alpha \in(0,1)$ when $f \in L^q(\Omega)$ with $q>d$. In addition, they have obtained boundary regularity of weak solutions to \eqref{MinFi}. We refer \cite[Section 1.1]{MG} for more details on the technical assumptions related to their operators. The work of \cite{MG} has been recently complemented by Biswas-Topp in \cite{AB1}, who obtained 
interior $\mathcal{C}^{1, \alpha}$-regularity for weak solutions to \eqref{MinFi} with $p\leq sr$ and locally bounded $f$. Finally, in \cite{antonini}, Antonini-Cozzi establish $\mathcal{C}^{1, \alpha}$-regularity for weak solutions, up to the boundary, with $f \in L^q(\Omega)$ with $q>d$. 
A general Hopf lemma is also obtained in \cite{antonini}. For further interior and boundary regularity estimates for weak solutions to the mixed local-nonlocal operators, we refer to \cite{ab3,harshni,globalgrad, nonstandard, fineboundary, valdinoci, valdinoci2} and the references therein. In \cite{garain2025two}, Garain recently established the Harnack inequality for weak solutions and the weak Harnack inequality for weak supersolutions to \eqref{MinFi} in the case $r=p$, assuming $f\in L^{q/p}(\Omega)$ for some $q>d$. Two alternative proofs were provided: one based on the classical John-Nirenberg lemma (Moser's approach), and the other on the Bombieri-Giusti lemma, combining inverse estimates and reverse Hölder inequalities for weak supersolutions, logarithmic estimates, and appropriate tail estimates. Finally, we mention the work of \cite{garainquali}, where, using De Giorgi's approach (as in \cite{KuusiHarnack}), the same author obtained Harnack and weak Harnack inequalities for weak solutions to \eqref{main_PDE} with $V \in L^{\infty}(\Omega)$.

In this paper, we prove the following regularity properties for weak solutions, weak subsolutions, and weak supersolutions (Definition \ref{maindef}) of \eqref{main_PDE}:

\begin{itemize}

\item \textbf{Local boundedness}. In Lemma \ref{local-boundedness-I}, we show that every weak subsolution to \eqref{main_PDE} is locally bounded. Our proof proceeds via an energy estimate (Lemma \ref{energy estimate}) combined with the Sobolev inequality and a Moser-type iteration scheme (Lemma \ref{iteration}). We emphasize that a direct adaptation of the arguments used in \cite{KuusiHarnack,garainjuha} is not feasible in our case due to the presence of an unbounded weight function $V$. By invoking the scaling properties detailed in Remark \ref{scale}, we may assume without loss of generality that, $u$ weakly solves the scaled equation $-\De_p u + \theta(-\De_p)^s u = V\abs{u}^{p-2}u \text{ in } \Omega$, where $\theta \in (0,1]$ and  $\norm{V}_{L^q(\Omega)}$ is sufficiently small. This smallness condition is essential for maintaining control over the constants within the estimates. Furthermore, the statement of local boundedness (see \eqref{loc.bounded}) introduces a parameter $\sigma$, where the prescribed integrability of $V$ ensures the strict inequalities $\sigma>1$ and $p\sigma<p^*$, which are vital for the convergence in the iteration process. 
Notably, if $V$ lies in the scaling critical space $L^{d/p}(\Omega)$, then $p \sigma = p^*$ causes the iteration process to fail, thereby precluding the derivation of local boundedness via this framework.  

\item \textbf{Harnack inequality and weak Harnack inequality}. In Theorem \ref{harnack_intro} we derive the Harnack inequality for weak solutions to \eqref{main_PDE}, and in Theorem \ref{weakhar_intro} we establish the weak Harnack inequality for weak super solutions to \eqref{main_PDE}. Our argument follows the strategy developed by Di Castro, Kuusi, and Palatucci \cite{KuusiHarnack}. In this approach, both the local boundedness and logarithmic energy estimate (Lemma \ref{log-estimate}) are fundamental components. Additionally, the expansion of positivity (Lemma \ref{expan}) and a suitable tail estimate (Lemma \ref{Tail}) play essential roles in the proof. The exponent $\sigma$ is again a key ingredient in establishing the expansion of positivity. The smallness assumption on $\|V\|$ and the strict inequality $p\sigma <p^*$ are also required at this stage.
\end{itemize}

Our results are applicable to a broader class of nonlocal operators defined by
$$
\mathcal{L} u(x)=\text{P.V.} \int_{\mathbb{R}^d} K_{\text {sym }}(x, y)|u(x)-u(y)|^{p-2}(u(x)-u(y)) \dy, \quad x \in \mathbb{R}^d ;
$$
where $K$ is a suitable kernel of order $(s, p)$ with merely measurable coefficients. The function $K_{\text {sym }}$ is the symmetric part of $K$ defined as $K_{\text {sym }}(x, y)=(K(x, y)+K(y, x)) / 2$, where  $K: \mathbb{R}^d \times \mathbb{R}^d \rightarrow[0, \infty)$ is a measurable function such that
$$
\lambda \leq K(x, y)|x-y|^{d+sp} \leq \Lambda, \text { for a.e. } x, y \in \mathbb{R}^d,
$$
where $\lambda \geq \Lambda \geq 1$. Further, the assumption on $K$ can be weakened as follows
$$
\begin{array}{ll}
\lambda \leq K(x, y)|x-y|^{d+s p} \leq \Lambda & \text { for a.e. } x, y \in \mathbb{R}^d \text { s.t. }|x-y| \leq 1, \\
0 \leq K(x, y)|x-y|^{d+\eta} \leq M & \text { for a.e. } x, y \in \mathbb{R}^d \text { s.t. }|x-y|>1,
\end{array}
$$
for some $\lambda, \Lambda$ as above, $\eta>0$ and $M \geq 1$.

\begin{remark}
  We do not address local H\"older continuity in this work, since it follows from \cite[Theorem 1.4]{garain2023higher} (at least in the case $p\geq 2$) with the fact that every weak solution to \eqref{main_PDE} is in $L^{\infty}(\Omega)$.   
\end{remark}
The rest of the paper is organized as follows. In Section \ref{prelims}, we introduce the relevant function spaces and present the main results. Section \ref{energylog} is devoted to establishing energy and logarithmic estimates for weak solutions to the scaled equation \eqref{eqn-scaled}. In Section \ref{localbdd}, we investigate the local boundedness of these weak solutions within our framework. Finally, Section \ref{harnackinq} provides the proofs of the expansion of positivity, the Harnack and weak Harnack inequalities.

\section{Function spaces and Main results}{\label{prelims}}
For an open set $E \subset \rd$, the Sobolev space $W^{1,p}(E)$  and the fractional Sobolev space $W^{s,p}(E)$ are defined as 
\begin{align*}
&W^{1,p}(E) := \left\{ u \in L^p(E) : \int_{E} \abs{ \Gr u(x)}^p \dx < \infty \right\},  \\
&W^{s,p}(E):= \left\{ u \in L^p(E) : \iint_{E \times E} \frac{|u(x)-u(y)|^p}{|x-y|^{d+sp}} \dxy < \infty \right\},
\end{align*}
which are endowed with the following norms respectively:
\begin{align*}
    &\norm{u}_{W^{1,p}(E)}:= \left( \int_{E} \abs{u(x)}^p \dx + \int_{E} \abs{\Gr u(x)}^p \dx \right)^{\frac{1}{p}}, \\
    &\norm{u}_{W^{s,p}(E)} := \left( \int_{E} \abs{u(x)}^p \dx + \iint_{E \times E} \frac{|u(x)-u(y)|^p}{|x-y|^{d+sp}} \dxy \right)^{\frac{1}{p}}. 
\end{align*}
Then we consider the Sobolev space $W_0^{1,p}(E)$, defined as $W_0^{1,p}(E) := \{ u \in W^{1,p}(\rd) : u=0 \text{ in } \rd \setminus E\}$.
From \cite[Proposition 2.2]{Hitchhiker}, the following continuous embedding holds: 
\begin{align*}
    \norm{u}_{W^{s,p}(E)} \le C(d,p,s) \norm{u}_{W^{1,p}(E)}, \; \forall \, u \in W^{1,p}(E).
\end{align*}

Now we recall the following Gagliardo-Nirenberg-Sobolev inequality, see \cite[Corollary 1.57]{MalyZiemer}.

\begin{lemma}\label{c.omega_sobo}
Let $1<p<\infty$ and $E$ be an open set in $\R^d$ with $\lvert E\rvert<\infty$ and
\begin{equation}\label{kappa}
\kappa=
\begin{cases}
\frac{d}{d-p},&\text{if}\quad d>p,\\
2,&\text{if}\quad d \leq p.
\end{cases}
\end{equation}
Then there exists a positive constant $C=C(d,p)$ such that
\begin{equation}\label{e.friedrich}
\biggl(\int_E \lvert u(x)\rvert^{\kappa p} \dx\biggr)^{\frac{1}{\kappa p}}
\le C(d,p) \lvert E \rvert^{\frac{1}{d}-\frac{1}{p}+\frac{1}{\kappa p}} \norm{u}_{W^{1,p}(E)},
\end{equation}
for every $u\in W^{1,p}(E)$.
\end{lemma}

In the study of nonlocal equations, the global behaviour of solutions comes into play. This is entailed by the {\it tail space}
\[
L^{q}_{\alpha}(\rd)=\left\{u\in L^{q}_{\rm loc}(\rd):\int_{\rd} \frac{|u|^q}{1+|x|^{d+\alpha}}\,\dx<+\infty\right\},\; q>0 \mbox{ and } \alpha>0,
\]
and measured by the quantity
\begin{equation}\label{mtail}
\mathrm{Tail}_{q,\alpha,\beta}(u;x_0,R)=\left(R^{\beta}\,\int_{\rd\setminus B_R(x_0)} \frac{|u|^q}{|x-x_0|^{d+\alpha}}\dx\right)^\frac{1}{q},
\end{equation}
defined for every $x_0\in\rd$, $R>0,\,\beta>0$ and $u\in L^q_{\alpha}(\rd)$. We observe that \eqref{mtail} is always finite, for a function $u\in L^q_{\alpha}(\rd)$.

Next, we recall the definitions of supersolution, subsolution, and solution of \eqref{main_PDE}.

\begin{definition}{\label{maindef}}
A function $u\in W^{1,p}_{loc}(\Omega) \cap L_{sp}^{p-1}(\rd)$ is a weak supersolution and subsolution to \eqref{main_PDE} if for every $K \subset \subset \Omega$ and $v\in W_0^{1,p}(K)$ with $v\geq0$ a.e. in $K$, it holds
\begin{align}\label{weak1}
    \int_{\Omega}|\nabla u|^{p-2}\nabla u\cdot\nabla v \,\dx+ \A(u,v) \geq \text{ or } \leq \int_{\Omega} V|u|^{p-2}uv \dx.
\end{align}
We say $u$ is a weak solution if the equality holds in \eqref{weak1} for every $v\in W_0^{1,p}(K)$.
\end{definition}

The main results of this paper are stated below. 

\begin{theorem}[Harnack Inequality]{\label{harnack_intro}}
Let $s \in (0,1), p \in (1, \infty)$, and $V \in L^q(\Omega)$ for $q$ as given in \eqref{weight}. Let  $R >0$ and $x_0 \in \Omega$ be such that $B_{R}(x_0)\subset \Omega$, and let $u$ be a weak solution of \eqref{main_PDE} satisfying $u \ge 0$ in $B_R(x_0)$. Then there exists $R_0 < R$ such that for
$r \in (0, \min\{\frac{R_0}{2},1\}]$, the following holds
\begin{equation*}
    \underset{B_{\frac{r}{2}}(x_0)}{\esssup}\, u \le C \,\underset{B_r(x_0)}{\essinf}\, u + C\left(\frac{r}{R_0}\right)^{\frac{p}{p-1}} \Tail_{p-1,sp,p}(u^-; x_0, R_0),
\end{equation*}
where $C=C(d,p,s)>0$ is a constant.  
\end{theorem}

\begin{theorem}[Weak Harnack Inequality]{\label{weakhar_intro}} 
Let $s \in (0,1), p \in (1, \infty)$, and $V \in L^q(\Omega)$ for $q$ as given in \eqref{weight}. Let  $R >0$ and $x_0 \in \Omega$ be such that $B_{R}(x_0)\subset \Omega$, and let $u$ be a weak supersolution of \eqref{main_PDE} satisfying $u \ge 0$ in $B_R(x_0)$. Then there exists $R_0 < R$ such that for
$r \in (0, \min\{\frac{R_0}{2},1\}]$, the following holds
$$
\left(\fint_{B_{\frac r2}\left(x_0\right)} u^l \dx\right)^{\frac{1}{l}} \leq C\, \underset{B_r(x_0)}{\essinf}\, u+C\left(\frac{r}{R_0}\right)^{\frac{p}{p-1}} \Tail_{p-1,sp,p}\left(u^{-} ; x_0, R_0\right),
$$
whenever $0<l<\kappa(p-1)$, with $\kappa$ as given in \eqref{kappa}. Here $C=C(d,p,s)>0$ is a constant. 
\end{theorem}

We require the following remark.
\begin{remark}\label{scale}
(a) Take $\rho \in (0,1)$ and $x_0 \in \Omega$ such that $B_{\rho R}(x_0) \subset \Omega$. By translation invariance of the operator, we may assume that $x_0 =0$. Then for a weak solution $u$ of \eqref{main_PDE}, the following holds weakly
$$-\De_p u + (-\De)_p^s u = V \abs{u}^{p-2}u, \text{ in } B_{\rho R}.$$  
Define $u_\rho(x) := u(\rho x)$. Then $u_\rho$ satisfies the following equation weakly
\begin{align*}
  -\De_p u_{\rho} + \rho^{p-sp} (-\De)_p^s u_\rho & = \rho^p \left( -\De_p u(\rho x) + (-\De)_p^s u(\rho x) \right) \\
  &= \rho^p V(\rho x) \abs{u(\rho x)}^{p-2}u(\rho x) = V_\rho(x) \abs{u_\rho}^{p-2} u_\rho, \text{ in } B_R,   
\end{align*}
where $V_{\rho}(x) = \rho^p V(\rho x)$. Moreover, for $q$ as given in \eqref{weight}, 
\begin{align*}
  \norm{V_\rho}_{L^{q}(B_R)} = \rho^{p-\frac{d}{q}} \norm{V}_{L^{q}(B_{\rho R})} = o(\rho),  
\end{align*}
as $\rho \ra 0$.

\noi (b) Now we consider $\tilde{\Omega} := B_R $. 
%Since $0 \in \Omega$, we observe that  $|\tilde{\Omega}| >0$. 
From (a), $u_{\rho}$ weakly solves the following equation
\begin{align*}
    -\De_p u_{\rho} + \rho^{p-sp} (-\De)_p^s u_\rho = V_\rho(x) \abs{u_\rho}^{p-2} u_\rho, \text{ in } \tilde{\Omega},
\end{align*}
where $\norm{V_\rho}_{L^q(\tilde{\Omega})}$ is sufficiently small depending on a fixed postive value of $\rho$.  
\end{remark}

In view of the above remark, from now onward, we consider the following equation  
\begin{align}\label{eqn-scaled}\tag{$\mathcal{P}_{\theta}$}
     -\De_p u + \theta (-\De)_p^s u = V \abs{u}^{p-2}u, \text{ in } \Omega; \quad \theta \in (0,1], 
\end{align}
with sufficiently small $\|V\|_{L^q(\Omega)}$ where $q$ as given in \eqref{weight}. We aim to establish the Harnack and weak Harnack estimates for weak solutions and weak super solutions to \eqref{eqn-scaled}, as presented in theorems \ref{harnack} and \ref{weakharnack}, respectively.

\noi \textbf{Notation and Convention.} We fix the following notations and conventions to be used in this paper:
\begin{enumerate}[(a)]
    \item We denote 
    \begin{align*}
       &\dm := |x-y|^{-(d+sp)} \dxy, \;  A_u(x,y):= |u(x)-u(y)|^{p-2}(u(x)-u(y)), \text{ and } \\
       &\mathcal{A}(u,v):= \iint_{\rd\times \rd}A_u(x,y)(v(x)-v(y)) \dm.
    \end{align*}
    \item For $d>p$, $p^*=\frac{d p}{d-p}$ is the critical Sobolev exponent. 
    \item $u^\pm := \max\{\pm u, 0\}$ denote the positive and negative parts of $u$.
    \item We denote $\Tail_{p-1,sp,p}(\cdot ;\cdot, \cdot)$ by $\Tail(\cdot ;\cdot, \cdot)$.
    \item We always take $x_0 \in \Omega$ and $r>0$ such that $B_r(x_0) \subset \Omega$. For brevity, we denote $B_r(x_0)$ as $B_r$.
    \item $C,C_i$ (where $i=1,2, \cdots$) denote generic positive constants.
    %\item $B_r(x_0)$ denotes a ball of radius $r$ with center at $x_0 \in \Omega$. %For brevity, we denote $B_r:= B_r(0)$.
    \item  Throughout the paper, we consider the case $d>p$. For $d\le p$, proof follows using similar set of arguments. 
\end{enumerate}

\section{Energy estimates and Logarithmic estimates}{\label{energylog}}

We begin with the following energy estimate for weak solutions to \eqref{eqn-scaled}.

\begin{lemma}\label{energy estimate}
   Let $V \in L_{loc}^1(\Omega)$, and $u$ be a weak subsolution to \eqref{eqn-scaled}. Let $w=(u-k)^+$ for $k \in \R$. Then there exists $C=C(p)$ such that 
   \begin{align}{\label{cacc}}
       &\int_{B_{r}(x_0)} \phi^{p} |\nabla w|^{p} \dx + \theta \iint_{B_{r}(x_0) \times B_{r}(x_0)} |w(x)\phi(x) - w(y)\phi(y)|^{p} \dm \nonumber\\ & \le C \Bigg(\int_{B_{r}(x_0)} w^{p} |\nabla \phi|^{p} \dx + \int_{B_{r}(x_0)}| V(x)| |u|^{p-1} w \phi^p \dx  \nonumber \\ & \quad + \theta \iint_{B_{r}(x_0) \times B_{r}(x_0)} \max\{w(x), w(y)\}^{p} |\phi(x)-\phi(y)|^{p} \dm  \nonumber \\
       &\quad + \theta \left(\underset{x \in \text{supp} (\phi)}{\esssup}\, \int_{\rd \setminus B_{r}(x_0)} \frac{w(y)^{p-1}}{|x-y|^{d+ps}} \dy \right) \int_{B_{r}(x_0)} w\phi^{p} \dx \Bigg),
   \end{align}
  where $\phi \in \C_c^{\infty}(B_r)$ is a nonnegative function. If $u$ is a
weak supersolution of \eqref{eqn-scaled}, the estimate in \eqref{cacc} holds with $w=(u-k)^-$.
\end{lemma}

\begin{proof}
Since $u$ is a subsolution to \eqref{eqn-scaled}, taking $v:= w \phi^{p} \in W_{0}^{1,p}(\Omega)$ as a nonnegative test function, we write
\begin{equation*}
\int_{\Omega} |\nabla u|^{p-2} \nabla u \cdot \nabla v \dx + \theta \A(u, v) \le \int_{\Omega} V|u|^{p-2}uv \dx.
\end{equation*} 
Using the above inequality, we obtain 
\begin{align}\label{CC1}
\int_{\Omega} |\nabla w|^{p} \phi^{p} \dx + \theta\A(u, v) & = \int_{\Omega} |\nabla u|^{p-2} \nabla u \cdot \nabla v \dx + \theta\A(u,v) - p \int_{\Omega} w \phi^{p-1} |\nabla w|^{p-2} \nabla w \cdot \nabla \phi \dx \nonumber \\ 
& \le \int_{\Omega} |V||u|^{p-1}v \dx - p \int_{\Omega} w \phi^{p-1} |\nabla w|^{p-2} \nabla w \cdot \nabla \phi  \dx.
\end{align}
From the nonlocal energy estimate \cite[Theorem 1.4]{DKP2016}, we have 
\begin{align}\label{CC2}
 \A(u, v) & \ge C(p) \iint_{B_{r} \times B_{r}} |w(x)\phi(x) - w(y)\phi(y)|^{p} \dm - C(p) \iint_{B_{r} \times B_{r}} \max\{w(x), w(y)\}^{p} |\phi(x)-\phi(y)|^{p} \dm \nonumber  \\
& - C(p) \underset{x \in \text{supp}(\phi)}{\esssup}\, \int_{\rd \setminus B_{r}} \frac{w(y)^{p-1}}{|x-y|^{d+ps}} \dy \cdot \int_{B_{r}} w\phi^{p} \dx.
\end{align}
Further, using Young's inequality, for $\var>0$, 
\begin{align*}
p \left| \int_{\Omega} w \phi^{p-1} |\nabla w|^{p-2} \nabla w \cdot \nabla \phi \dx \right| & \le p \int_{\Omega} |w| \phi^{p-1} |\nabla w|^{p-1} |\nabla \phi| \dx \nonumber \\
&  \le \varepsilon \int_{B_{r}} |\nabla w|^{p} \phi^{p} \dx + O\left(\frac{1}{\varepsilon}\right) C(p) \int_{B_{r}} |\nabla \phi|^{p} w^{p} \dx.
\end{align*}
By taking $\varepsilon = (2(2^{p}+1))^{-1}$, and combining \eqref{CC1} and \eqref{CC2}, the required estimate holds. In the case of a weak supersolution, the estimate in \eqref{cacc} follows by applying the obtained result to $-u$.
\end{proof}

The following lemma establishes an energy estimate for a weak supersolution to \eqref{eqn-scaled}, which will be helpful in proving the weak Harnack inequality. 

\begin{lemma}{\label{3.3}} 
Assume that $u$ is a weak supersolution to \eqref{eqn-scaled} satisfying $u \geq 0$ in $B_R(x_0) \subset \Omega$.  Let $B_r (x_0)\subset B_{\frac{3R}{4}}(x_0)$. Denote $w=(u+t)^{\frac{p-m}{p}}$ where $m \in(1, p)$ and $t>0$. Then there exists $\zeta_1(m,d,p)>0$ such that if $\|V\|_{L^q(\Omega)} \le \zeta_1$, then there exists $C=C(p)$ such that
\begin{align*}
\int_{B_r(x_0)} \phi^p|\nabla w|^p \dx \leq C & \left(\frac{(p-m)^p}{(m-1)^{\frac{p}{p-1}}} \int_{B_r(x_0)} w^p|\nabla \phi|^p \dx\right. \\
& +\theta\frac{(p-m)^p}{(m-1)^p} \iint_{B_r(x_0) \times B_r(x_0)} \max \{w(x), w(y)\}^p|\phi(x)-\phi(y)|^p \dm \\
& +\theta\frac{(p-m)^p}{(m-1)}\bigg(\underset{z \in \operatorname{supp} \phi}{\esssup}\, \int_{\mathbb{R}^d \backslash B_r(x_0)} \frac{\dy}{|z-y|^{d+ps}} \\
& \left.+t^{1-p} R^{-p} \operatorname{Tail}\left(u^{-} ; x_0, R\right)^{p-1}\bigg) \int_{B_r(x_0)} w^p \phi^p \dx \right),  
\end{align*}
for every $\phi \in \C_c^{\infty}(B_r)$ with $\phi \ge 0$. 
\end{lemma}

\begin{proof} 
Let $t>0, h=u+t$ and $m \in[1+\varepsilon, p-\varepsilon]$ for $\varepsilon>0$ small enough. As $u$ is a weak supersolution of \eqref{eqn-scaled}, by choosing $h^{1-m} \phi^p$ as a test function in \eqref{eqn-scaled}, we obtain
\begin{align}{\label{1}}
& 0 \leq\int_{B_r}|\nabla u|^{p-2} \nabla u \cdot \nabla(h^{1-m} \phi^p) \dx \no\\
&+\theta\iint_{B_r \times B_r} |u(x)-u(y)|^{p-2}(u(x)-u(y))(h(x)^{1-m} \phi(x)^p-h(y)^{1-m} \phi(y)^p) \dm\no \\
&+2 \theta\iint_{\mathbb{R}^d \backslash B_r \times B_r} |u(x)-u(y)|^{p-2}(u(x)-u(y))h(x)^{1-m} \phi(x)^p \dm+\int_{B_{r}} |V| \frac{u^{p-1} \phi^p}{(u+t)^{m-1}} \dx \no\\&= \int_{B_r}|\nabla h|^{p-2} \nabla h \cdot \nabla(h^{1-m} \phi^p) \dx \no\\
&+\theta\iint_{B_r \times B_r} |h(x)-h(y)|^{p-2}(h(x)-h(y))(h(x)^{1-m} \phi(x)^p-h(y)^{1-m} \phi(y)^p) \dm\no \\
&+2\theta \iint_{\mathbb{R}^d \backslash B_r \times B_r} |h(x)-h(y)|^{p-2}(h(x)-h(y))h(x)^{1-m} \phi(x)^p \dm+\int_{B_{r}} |V| \frac{u^{p-1} \phi^p}{(u+t)^{m-1}} \dx \no\\
&=:I_1+ \theta I_2+2\theta I_3+I_4.
\end{align}

\noi \textbf{Estimate of $I_1$.} We observe that
\begin{align}{\label{I1}}
I_1 & =\int_{B_r}|\nabla h|^{p-2} \nabla h \cdot \nabla\left(h^{1-m} \phi^p\right) \dx \no\\
& \leq(1-m) \int_{B_r} h^{-m}|\nabla h|^p \phi^p \dx+p \int_{B_r} h^{1-m}|\nabla \phi||\nabla h|^{p-1} \phi^{p-1} \dx \no\\
& =(1-m) J_1+J_2,
\end{align}
where
$$
J_1=\int_{B_r} h^{-m}|\nabla h|^p \phi^p \dx,\; \text{ and }\; J_2=p \int_{B_r} h^{1-m}|\nabla \phi||\nabla h|^{p-1} \phi^{p-1} \dx.
$$
Now, by Young's inequality, we obtain
\begin{align}{\label{J2}}
J_2 
& \leq \frac{m-1}{2} J_1+\frac{C(p)}{(m-1)^{\frac{1}{p-1}}} \int_{B_r}|\nabla \phi|^p h^{p-m} \dx.
\end{align}
By applying \eqref{J2} in \eqref{I1}, for some constant $C=C(p)>0$, we have
\begin{align}{\label{I11}}
I_1 & \leq \frac{1-m}{2} \int_{B_r} h^{-m}|\nabla h|^p \phi^p \dx+\frac{C}{(m-1)^{\frac{1}{p-1}}} \int_{B_r}|\nabla \phi|^p h^{p-m} \dx \no\\
& =-\frac{m-1}{2}\left(\frac{p}{p-m}\right)^p \int_{B_r}\left|\nabla\left(h^{\frac{p-m}{p}}\right)\right|^p \phi^p \dx+\frac{C}{(m-1)^{\frac{1}{p-1}}} \int_{B_r}|\nabla \phi|^p h^{p-m} \dx.
\end{align}

\noi \textbf{Estimates of $I_2$ and $I_3$.} Following the lines of the proof of \cite[Lemma 5.1]{KuusiHarnack} for $w=h^{\frac{p-m}{p}}$, with some positive constants $C(p,m)$ and $C(p)$, we obtain
\begin{align}{\label{I2}}
& I_2+I_3 
\leq -C(p,m) \iint_{B_r \times B_r}|w(x)-w(y)|^p \phi(y)^p \dm \no\\
& +\frac{C(p)}{(m-1)^{p-1}} \iint_{B_r \times B_r} \max \{w(x), w(y)\}^p|\phi(x)-\phi(y)|^p \dm\no\\&+C\left(\underset{z \in \operatorname{supp} \phi}{\esssup}\, \int_{\mathbb{R}^d \backslash B_r} \frac{1}{|z-y|^{d+ps}}\dy+t^{1-p} \int_{\mathbb{R}^d \backslash B_r}\frac{((u(y))^{-})^{p-1}}{|y-x_0|^{d+ps}} \dy \right) \int_{B_r} w^p \phi^p \dx.
\end{align}

\noi \textbf{Estimate of $I_4$.} 
Noting $t>0$, and the fact that $\phi w=\phi(u+t)^{\frac{p-m}{p}}\in W^{1,p}_0(\Omega)$, by Sobolev and H\"older inequality, it holds 
\begin{align}{\label{I4}}
\int_{B_{r}} & |V| \frac{u^{p-1} \phi^p}{(u+t)^{m-1}} \dx\leq \int_{B_{r}} |V| u^{p-m} \phi^p\dx=\int_{B_{r}} |V| (\phi w)^p \dx\no\\
&\leq C(d,p) \|V\|_{L^{q}(B_r)}|B_{r}|^{\frac{p}{d}-\frac1q} \left\| \phi w \right\|_{L^{p^*}(B_{r})}^p \no\\
&\leq C(d,p) \|V\|_{L^{q}(B_r)} \int_{B_r} |\nabla(\phi w)|^p\dx \no\\
&\leq C(d,p) \|V\|_{L^{q}(B_r)} \int_{B_r} |\nabla w|^p\phi ^p\dx+C(d,p) \|V\|_{L^{q}(B_r)} \int_{B_r} w^p|\nabla\phi |^p\dx.
\end{align}
We now choose $\zeta_1=\zeta_1(d,p,m)>0$ in \eqref{I4} so small such that
$$
C(d,p) \zeta_1 \leq \min\left\{\frac{m-1}{4}\left(\frac{p}{p-m}\right)^p, \frac{1}{(m-1)^{\frac{1}{p-1}}}\right\}.
$$
Thus, whenever $\|V\|_{L^q(\Omega)}\leq \zeta_1$,
\begin{align}{\label{I44}}
I_4 
& \leq\frac{m-1}{4}\left(\frac{p}{p-m}\right)^p \int_{B_r}\left|\nabla\left(h^{\frac{p-m}{p}}\right)\right|^p \phi^p \dx+\frac{C}{(m-1)^{\frac{1}{p-1}}} \int_{B_r}|\nabla \phi|^p h^{p-m} \dx.
\end{align}
By applying \eqref{I11}, \eqref{I2} and \eqref{I44} in \eqref{1}, we conclude the proof.
\end{proof}

In the following lemma, we obtain a logarithmic energy estimate for weak supersolution.

\begin{lemma}\label{log-estimate} 
Assume that $u$ is a weak supersolution to \eqref{eqn-scaled} satisfying $u \geq 0$ in $B_R(x_0) \subset \Omega$. Let $B_r (x_0)\subset B_{\frac{R}{2}}(x_0)$ and $t>0$. Then there exists $C=C(d, p, s)$ such that the following holds:
\begin{align*}
&\int_{B_r(x_0)}|\nabla \log (u+t)|^p \dx +\theta\iint_{B_r(x_0) \times B_r(x_0)}\left|\log \left(\frac{u(x)+t}{u(y)+t}\right)\right|^p \dm  \nonumber\\&\leq C r^d\left(r^{-p}+\theta r^{-p s}+\theta t^{1-p} R^{-p} \operatorname{Tail}\left(u^{-} ; x_0, R\right)^{p-1}\right).
\end{align*}
\end{lemma}

\begin{proof}Consider a cut-off function $\phi \in \C_c^{\infty}(B_{\frac{3 r}{2}})$ satisfying $\phi =
1$ on $B_r, 0 \leq \phi \leq 1$ and $|\nabla \phi| \leq \frac{C}{r}$ on $B_{\frac{3 r}{2}}$. By choosing $v=(u+t)^{1-p} \phi^p \in W_0^{1, p}(\Omega)$ as a nonnegative test function we obtain
\begin{align}{\label{log}}
0 &\leq \int_{B_{2 r}}|\nabla u|^{p-2} \nabla u \cdot \nabla\left(\frac{\phi^p}{(u+t)^{p-1}}\right) \mathrm{d} x+\theta\iint_{\mathbb{R}^d \times \mathbb{R}^d}|u(x)-u(y)|^{p-2}(u(x)-u(y)) \nonumber\\
&\quad\times\left(\frac{\phi(x)^p}{(u(x)+t)^{p-1}}-\frac{\phi(y)^p}{(u(y)+t)^{p-1}}\right) \mathrm{d} \mu+\int_{B_{2 r}} |V| \frac{u^{p-1} \phi^p}{(u+t)^{p-1}} \dx=:I_1+\theta I_s+I_2.
\end{align}
Let $\varepsilon \in(0,1)$ be given. Applying the Young's inequality, we estimate $I_1$ as follows:
\begin{align}{\label{I_1}}
I_1 & =\int_{B_{2 r}}|\nabla u|^{p-2} \nabla u \cdot\left((1-p) \frac{\phi^p \nabla u}{(u+t)^p}+p \frac{\phi^{p-1} \nabla \phi}{(u+t)^{p-1}}\right) \mathrm{d} x \nonumber\\
& \leq(1-p+\varepsilon) \int_{B_{2 r}} \frac{|\nabla u|^p \phi^p}{(u+t)^p} \dx+O\left(\frac{1}{\varepsilon}\right) C(p) \int_{B_{\frac{3 r}{2}}}|\nabla \phi|^p \dx.
\end{align}
By noticing the following that
\begin{align*}
|\nabla \log (u+t)|^p=\frac{|\nabla u|^p}{(u+t)^p}, \text { and } \int_{B_{\frac{3 r}{2}}}|\nabla \phi|^p \dx=C(d) \int_r^{\frac{3 r}{2}} \tau^{d-1-p} \mathrm{~d} \tau=C(d, p) r^{d-p}
\end{align*}
and using \eqref{I_1}, we obtain
\begin{align}{\label{log1}}
    I_1 \leq(1-p+\varepsilon) \int_{B_{2 r}}|\nabla \log (u+t)|^p \phi^p \dx+O\left(\frac{1}{\varepsilon}\right) C(d, p) r^{d-p}.
\end{align}
Now we estimate $I_2$ as follows:
\begin{align}{\label{log2}}
    I_2=\int_{B_{\frac{3 r}{2}}} |V| \frac{u^{p-1} \phi^p}{(u+t)^{p-1}} \dx& \leq \int_{B_{\frac{3 r}{2}}}|V| \phi^p \dx  \leq\left\|V\right\|_{L^{q}(\Omega)} |B_{\frac{3r}{2}}|^{\frac{p}{d}-\frac1q} \left\| \phi \right\|_{L^{p^*}(B_{\frac{3r}{2}})}^p \no \\
    &\leq C(d,p) \|V\|_{L^{q}(\Omega)}\int_{B_{\frac{3 r}{2}}}|\nabla \phi|^p \dx\leq C \int_r^{\frac{3 r}{2}} \tau^{d-1-p} \mathrm{~d} \tau=C r^{d-p} .
\end{align}
Therefore, for $\varepsilon=\frac{p-1}{2}$, using \eqref{log1}, \eqref{log2} we get
\begin{align}{\label{log3}}
    I_1+I_2 \leq-C(p) \int_{B_{2 r}}|\nabla \log (u+t)|^p \phi^p \dx+C(d, p) r^{d-p}.
\end{align}
Further, following \cite[Lemma 1.3]{KuusiHarnack}, for some $C=C(d, p, s)$ we obtain
\begin{align}{\label{log4}}
    I_s \leq-\frac{1}{C} \iint_{B_{2 r} \times B_{2 r}}\left|\log \left(\frac{u(x)+t}{u(y)+t}\right)\right|^p \phi(y)^p \mathrm{~d} \mu+C r^{d-p s}+C \frac{r^d}{t^{p-1} R^p} \operatorname{Tail}\left(u^{-} ; x_0, R\right)^{p-1}.
\end{align}Therefore, the conclusion follows by combining \eqref{log}, \eqref{log3} and \eqref{log4}.
\end{proof}

\section{Local boundedness and Tail estimates}{\label{localbdd}}

In this section, we establish that the weak solution to \eqref{eqn-scaled} is locally bounded. We then prove the Tail estimate for the weak solution. The following lemma is required for local boundedness. For the proof, see \cite[Lemma 4.1]{Dibe}.

\begin{lemma}[Iteration Lemma]\label{iteration}
Let $(Y_j)_{j=0}^{\infty}$ be a sequence of positive real numbers such that
\begin{align*}
    Y_0\leq c_{0}^{-\frac{1}{\beta}}b^{-\frac{1}{\beta^2}} \text{ and } Y_{j+1}\leq c_0 b^{j} Y_j^{1+\beta},
\end{align*}
$j=0,1,2,\dots$, for some constants $c_0,b>1$ and $\beta>0$. Then $Y_j \ra 0$ as $j \ra \infty$.
\end{lemma}

\begin{proposition}[Local Boundedness]\label{local-boundedness-I}
Assume that $u$ is a weak subsolution to \eqref{eqn-scaled} satisfying $u \geq 0$ in $B_R(x_0) \subset \Omega$. Let $0<r\leq1$ be such that $B_r(x_0) \subset B_R(x_0)$. Then there exists $\zeta_2(d,p)>0$ such that for $\|V\|_{L^q(\Omega)} \le \zeta_2$, 
\begin{equation}\label{loc.bounded}
    \underset{B_{\frac{r}{2}}(x_0)}{\esssup}\,u\leq \delta \theta^{\frac{1}{p-1}} \text{\rm Tail}(u^+;x_0,\frac r2)+C(d,s,p)\delta^{\frac{(1-p)p^*}{p(p^*-p\sig)}}\left(\fint_{B_r(x_0)}u^{p\sig}\dx\right)^{\frac1{p\sig}},
\end{equation}
where $\delta\in(0,1], \sig=\frac{(dp-d+p)q-d}{(dp-d+p)q-dp}$ and $p\sigma<p^*$.
\end{proposition}
\begin{proof}
For $r\in(0,1)$ and $j=0,1,2, \cdots$, define $$r_j=\frac{r}{2}(1+2^{-j}),\,\overline{r}_j=\frac{r_j+r_{j+1}}2,\,B_j=B_{r_j}(x_0),\, \text{ and } \,\overline{B}_j=B_{\overline{r}_j}(x_0).$$ Let $\{\phi_j\}_{j=1}^\infty\subset \C_c^{\infty}(\overline{B}_j)$ be a nonnegative sequence of cut-off functions such that $$0\leq \phi_j\leq 1\text{ in }\overline{B}_j,\,\phi_j\equiv1\text{ on }B_{j+1},\, |\nabla\phi_j|\leq \frac{2^{j+3}}{r}.$$
For $k,\overline{k}\geq0$, we denote
$$k_j=k+(1-2^{-j})\overline{k},\,\overline{k}_j=\frac{k_j+k_{j+1}}{2},\, w_j=(u-k_j)_+,\, \text{ and } \,\overline{w}_j=(u-\overline{k}_j)_+.$$
Using the energy estimate (Lemma \ref{energy estimate}), we obtain
\begin{equation}\label{caccioppoli-1}
\begin{aligned}
    &\int_{B_j}\phi_j^p|\nabla \overline{w}_j|^p\dx+\theta\iint_{B_j^2}|\overline{w}_j(x)\phi_j(x)-\overline{w}_j(y)\phi_j(y)|^p\dm\\
    &\leq C(p) \bigg(\int_{B_j}\overline{w}_j^p|\nabla\phi_j|^p\dx+\theta\iint_{B_j^2}\max\{\overline{w}_j(x),\overline{w}_j(y)\}^p \left| \phi_j(x)-\phi_j(y) \right|^p\dm\\
    &\quad\qquad+\theta \,\underset{x\in\supp\phi}{\esssup}\,\int_{\rd\setminus B_j}\frac{\overline{w}_j(y)^{p-1}}{|x-y|^{d+sp}}\dy\cdot\int_{B_j}\overline{w}_j\phi_j^p\dx\bigg)+\int_{B_j}V(x)u^{p-1}\overline{w}_j\phi_j^p\dx \bigg).
\end{aligned}
\end{equation}
Applying H\"{o}lder's inequality,
\begin{align}\label{lb-1}
    &\int_{B_j} V(x)u^{p-1}\overline{w}_j\phi_j^p\dx \le C(p) \left(\int_{B_j}\abs{V(x)}\overline{w}_j^p \phi_j^p \dx+\overline{k}_j^{p-1} \int_{B_j} |V(x)| \overline{w}_j\phi_j^p \dx \right) \no \\
    &\leq C(p) \left\|V\right\|_{L^{q}(\Omega)} |B_j|^{\frac{p}{d}-\frac1q} \left\| \overline{w}_j\phi_j \right\|_{L^{p^*}(B_j)}^p + \overline{k}_j^{p-1} \left\|V \phi_j^{p-1} \right\|_{L^{q}(B_j)} \left\|\overline{w}_j\phi_j\right\|_{L^{\frac{q}{q-1}}(B_j)} \no \\
    &\leq C(d,p) \|V\|_{L^{q}(\Omega)} \left(\left\|\overline{w}_j\phi_j\right\|_{L^{p^*}(B_j)}^p+\overline{k}_j^{p-1}\left\|\overline{w}_j\phi_j\right\|_{L^{\frac{q}{q-1}}(B_j)}\right).
\end{align}
Applying the generalized H\"{o}lder's inequality $\frac{1}{p_1} = \frac{\al}{p_2} + \frac{1-\al}{p_3}$ with the conjugate triplet $(p_1,p_2,p_3)$, where 
\begin{align*}
    p_1=\frac{q}{q-1}, p_2= p^*, p_3=1, \text{ and } \al=\frac{p^*}{(p^*-1)q},
\end{align*}
and observing the fact that 
\begin{align*}
    q>\frac{d}{p} \Longrightarrow \frac{q}{q-1} < \frac{d}{d-p} <p^* \Longrightarrow q > \frac{p^*}{p^*-1} \Longrightarrow \al < 1,
\end{align*}
we now estimate
\begin{align*}
    \left\|\overline{w}_j\phi_j \right\|_{L^{\frac{q}{q-1}}(B_j)}  \le \left\| \overline{w}_j \phi_j \right\|_{L^{p^*}(B_j)}^{\al} \left\|\overline{w}_j\phi_j\right\|_{L^1(B_j)}^{1-\al}.
\end{align*}
Now, applying the Young's inequality with coefficients $(\frac{p}{\al},\frac{p}{p-\al})$, we obtain 
\begin{align*}
\overline{k}_j^{p-1} \left\|\overline{w}_j\phi_j\right\|_{L^{p^*}(B_j)}^\al\,\left\|\overline{w}_j\phi_j\right\|_{L^{1}(B_j)}^{1-\al} \leq \|\overline{w}_j\phi_j\|_{L^{p^*}(B_j)}^p+\overline{k}_j^{\frac{p(p-1)}{p-\al}}\|\overline{w}_j\phi_j\|_{L^{1}(B_j)}^{\frac{p(1-\al)}{p-\al}}.
\end{align*}
Using the Sobolev embedding $W^{1,p}_0(B_j) \hookrightarrow L^{p^*}(B_j)$ and the fact that $\overline{w}_j\phi_j \in W^{1,p}_0(B_j)$, we get 
\begin{align*}
    \|\overline{w}_j\phi_j\|_{L^{p^*}(B_j)}^p \le C(d,p) \norm{\Gr (\overline{w}_j\phi_j)}_{L^{p}(B_j)}^p \le C(d,p) \left( \int_{B_j}\phi_j^p|\nabla \overline{w}_j|^p \dx + \int_{B_j}\overline{w}_j^p|\nabla\phi_j|^p\dx \right). 
\end{align*}
Therefore, noting \eqref{lb-1} and choosing $\zeta_2 =  \zeta_2(d,p) >0$ small enough, we get from \eqref{caccioppoli-1},
\begin{equation}\label{caccioppoli-2}
\begin{aligned}
    &\int_{B_j}\phi_j^p|\nabla \overline{w}_j|^p \dx+\theta\iint_{B_j^2}|\overline{w}_j(x)\phi_j(x)-\overline{w}_j(y)\phi_j(y)|^p\dm\\
    &\leq C\bigg(\int_{B_j}\overline{w}_j^p|\nabla\phi_j|^p\dx+\theta\iint_{B_j^2}\max\{\overline{w}_j(x),\overline{w}_j(y)\}^p|\phi_j(x)-\phi_j(y)|^p\dm\\
&\quad\qquad+ \theta \, \underset{x\in\supp\phi}{\esssup}\,\int_{\rd\setminus B_j}\frac{\overline{w}_j(y)^{p-1}}{|x-y|^{d+sp}}\dy\cdot\int_{B_j}\overline{w}_j\phi_j^p\dx+\overline{k}_j^{\frac{p(p-1)}{p-\al}}\|\overline{w}_j\phi_j\|_{L^{1}(B_j)}^{\frac{p(1-\al)}{p-\al}}\bigg).
\end{aligned}
\end{equation}
Again, using the Sobolev inequality, \eqref{caccioppoli-2} yields
\begin{align}\label{caccioppoli-3}
   &\left(\fint_{B_j}|\overline{w}_j \phi_j|^{p^*}\dx\right)^{\frac{p}{p^*}}\leq Cr^{p-d}\|\nabla(\overline{w}_j\phi_j)\|_{L^p(B_j)}^p\leq Cr^{p-d}\left(\int_{B_j}\phi_j^p|\nabla \overline{w}_j|^p\dx+\int_{B_j}\overline{w}_j^p|\nabla\phi_j|^p\dx\right) \no \\
    &\leq Cr^{p-d}\bigg(\int_{B_j}\overline{w}_j^p|\nabla\phi_j|^p\dx+\theta\iint_{B_j^2}\max\{\overline{w}_j(x),\overline{w}_j(y)\}^p|\phi_j(x)-\phi_j(y)|^p\dm \no \\
    &+\theta \,\underset{x\in\supp\phi}{\esssup}\,\int_{\rd\setminus B_j}\frac{\overline{w}_j(y)^{p-1}}{|x-y|^{d+sp}}\dy\cdot\int_{B_j}\overline{w}_j\phi_j^p\dx+\overline{k}_j^{\frac{p(p-1)}{p-\al}}\|\overline{w}_j\phi_j\|_{L^1(B_j)}^{\frac{p(1-\al)}{p-\al}}\bigg),
\end{align}
where $C=C(d,p,s)$. Using $|\nabla\phi_j|\leq \frac{2^{j+3}}{r}$, observe that 
\begin{align*}
    r^{p-d} \int_{B_j}\overline{w}_j^p|\nabla\phi_j|^p\dx \le C(d,p) 2^{jp} \fint_{B_j} {w}_j^p \dx.
\end{align*}
Further, we estimate
\begin{align*}
    &\theta r^p \fint_{B_j} \int_{B_j}\max\{\overline{w}_j(x),\overline{w}_j(y)\}^p|\phi_j(x)-\phi_j(y)|^p\dm \le C\theta 2^{jp}  \fint_{B_j} w_j^p(y) \left( \int_{B_j} \frac{\dx}{\abs{x-y}^{d+ps-p}} \right) \dy \\
    & \le \frac{C\theta 
2^{jp}}{p(1-s)} r^{p-ps}\fint_{B_j} w_j^p(x) \dx\le C2^{jp}\fint_{B_j} w_j^p(x) \dx. 
\end{align*}
Proceeding similarly as in the proof of \cite[Page 1292]{DKP2016}, we get 
\begin{align*}
    & \theta r^{p-d} \, \underset{x\in\supp\phi}{\esssup}\,\int_{\rd\setminus B_j}\frac{\overline{w}_j(y)^{p-1}}{|x-y|^{d+sp}}\dy\cdot\int_{B_j}\overline{w}_j\phi_j^p\dx \\
    &\le C(d,s,p) 2^{j(d+sp+p-1)} \theta  \left( \frac{\text{Tail}(u^+;x_0,\frac r2)}{\overline{k}} \right)^{p-1} \fint_{B_j} w_j^p \dx.
\end{align*} 
Also, observe that 
$$\fint_{B_j}w_j^p\dx\leq \left(\fint_{B_j}w_j^{p\sig}\right)^{\frac1\sig}.$$
Next, we estimate the last integral of \eqref{caccioppoli-3}.
Define $\sig:=\frac{p-\al}{p(1-\al)}>1$. Then $p\sig\in(p,p^*)$. Using the relations 
\begin{align*}
    \frac{p(p-1)}{p- \alpha} = p - \frac{1}{\sigma}, \overline{w}_j\leq(\overline{k}_j-k_j)^{1-p\sig}w_j^{p\sig}, \text{ and } \overline{k}_j\leq k+\overline{k},
\end{align*}
we get
\begin{align*}
     \overline{k}_j^{\frac{p(p-1)}{p-\al}}\|\overline{w}_j\phi_j\|_{L^1(B_j)}^{\frac1\sig}\leq\left(\frac{k+\overline{k}}{\overline{k}_j-k_j}\right)^{p-\frac1\sig}\left(\int_{B_j}w_j^{p\sig}\right)^{\frac1\sig} \le \left(2^{j+2} \frac{k+\overline{k}}{\overline{k}}\right)^{p-\frac1\sig} r^{\frac{d}{\sigma}}\left(\fint_{B_j}w_j^{p\sig}\right)^{\frac1\sig}.
\end{align*}
Therefore, from \eqref{caccioppoli-3}, there exists $C=C(d,p,s)$ such that for every large $j$, 
\begin{align*}
\left(\fint_{\overline{B}_j}|\overline{w}_j\phi_j|^{p^*}\right)^{\frac{p}{p^*}}&\leq C \left[2^{jp}+2^{j(d+sp+p-1)} \theta \left( \frac{\text{Tail}(u^+;x_0,\frac r2)}{\overline{k}} \right)^{p-1} +r^{p-d+\frac{d}{\sigma}}\left(2^{j+2}\frac{k+\overline{k}}{\overline{k}}\right)^{p-\frac1\sig}\right] \\
&\quad \left(\fint_{B_j}w_j^{p\sig}\right)^{\frac1\sig} \\
& \le C 2^{j(d+sp+p-1)} \left[\theta \left( \frac{\text{Tail}(u^+;x_0,\frac r2)}{\overline{k}} \right)^{p-1} +r^{p-d+\frac{d}{\sigma}}\left(\frac{k}{\overline{k}}\right)^{p-\frac1\sig}+1\right] \left(\fint_{B_j}w_j^{p\sig}\right)^{\frac1\sig},
\end{align*}
where using the definition of $\sigma$, we observe that $p-d+\frac d\sigma\geq 0$.
For $\delta\in(0,1]$, we let $\overline{k}$ such that 
\begin{align*}
    \theta \left( \frac{\text{Tail}(u^+;x_0,\frac r2)}{\overline{k}} \right)^{p-1} \le \delta^{1-p} \Longleftrightarrow \overline{k} \ge \delta \theta^{\frac{1}{p-1}}\text{Tail}(u^+;x_0,\frac r2),
\end{align*}
and 
\begin{align*}
    r^{p-d+\frac{d}{\sigma}}\left(\frac{k}{\overline{k}}\right)^{p-\frac1\sig} \le \delta^{1-p} \Longleftrightarrow \overline{k} \ge kr^{\left(p-d+\frac{d}{\sigma} \right)\left( \frac{p-\alpha}{p(p-1)}\right)} \delta^{\frac{p-\alpha}{p}}. 
\end{align*}
In view of the above inequalities, we take 
$$\overline{k}\geq \delta \theta^{\frac{1}{p-1}}\text{Tail}(u^+;x_0,\frac r2)+kr^{\left(p-d+\frac{d}{\sigma} \right)\left( \frac{p-\alpha}{p(p-1)}\right)} \delta^{\frac{p-\alpha}{p}},$$
so that 
\begin{align}\label{lb-3}
\left(\fint_{\overline{B}_j}|\overline{w}_j\phi_j|^{p^*}\right)^{\frac{p}{p^*}}&\leq C 2^{j(d+sp+p-1)}\delta^{1-p}\left(\fint_{B_j}w_j^{p\sig}\right)^{\frac1\sig}.
\end{align}
We also have 
\begin{align}\label{lb-4}
\left(\fint_{\overline{B}_j}|\overline{w}_j\phi_j|^{p^*}\right)^{\frac{p}{p^*}}&\geq (k_{j+1}-\overline{k}_j)^{\frac{p(p^*-p\sig)}{p^*}}\left(\fint_{B_{j+1}}w_{j+1}^{p\sig}\right)^{\frac{p}{p^*}} \no \\
&= \left(\frac{\overline{k}}{2^{j+2}}\right)^{\frac{p(p^*-p\sig)}{p^*}}\left(\fint_{B_{j+1}}w_{j+1}^{p\sig}\right)^{\frac{p}{p^*}}.
\end{align}
Denote $Y_j:=\displaystyle \left(\fint_{B_j}w_j^{p\sig}\right)^{\frac1{p\sig}}.$ Then using \eqref{lb-3} and \eqref{lb-4}, there exists $C=C(d,p,s)>1$ such that 
\begin{align*}
    \left(\frac{\overline{k}}{2^{j}}\right)^{\frac{p(p^*-p\sig)}{p^*}}Y_{j+1}^{\frac{p^2\sig}{p^*}}\leq C 2^{j(d+sp+p-1)} \delta^{1-p} Y_j^p,
\end{align*}
which implies
\begin{align*}
    \frac{Y_{j+1}}{\overline{k}}\leq C \delta^{\frac{p^*(1-p)}{p^2 \sigma}} \tilde{C}^j\left(\frac{Y_j}{\overline{k}}\right)^{1+\beta},
\end{align*}
where using the fact $p\sig <p^*$, we see that $\beta:=\frac{p^*}{p\sig}-1>0$ and $\tilde{C}:=2^{\left(\frac{d+sp+p-1}{p}+\frac{p^*-p\sig}{p^*}\right)\frac{p^*}{p\sig}}>1.$
Finally, we choose $$\overline{k}:=\delta \theta^{\frac{1}{p-1}}\text{Tail}(u^+;x_0,\frac r2)+kr^{\left(p-d+\frac{d}{\sigma} \right)\left( \frac{p-\alpha}{p(p-1)}\right)} \delta^{\frac{p-\alpha}{p}}+\delta^{\frac{(1-p)p^*}{p(p^*-p\sig)}}C^{\frac{1}{\be}}\tilde{C}^{\frac{1}{\be^2}}Y_0,$$ so that 
$$\frac{Y_0}{\overline{k}}\leq \delta^{\frac{p^*(p-1)}{p^2\sig\be}}C^{-\frac{1}{\be}} \tilde{C}^{-\frac{1}{\be^2}} = \left( C \delta^{\frac{p^*(1-p)}{p^2\sig}} \right)^{-\frac{1}{\beta}} \tilde{C}^{-\frac{1}{\be^2}}. $$
Now we take $c_0 = C\delta^{\frac{p^*(1-p)}{p^2\sig}}$  and $b=\tilde{C}$ in Lemma \ref{iteration}.
Therefore, applying the iteration lemma (Lemma \ref{iteration}), we get $Y_j\to0\text{ as }j\to\infty.$
Hence
\begin{align*}
    \sup_{B_{\frac{r}{2}}(x_0)}(u-k)^+\leq \overline{k} & \leq\delta \theta^{\frac{1}{p-1}}\text{Tail}(u^+;x_0,\frac r2)+kr^{\left(p-d+\frac{d}{\sigma} \right)\left( \frac{p-\alpha}{p(p-1)}\right)} \delta^{\frac{p-\alpha}{p}} 
    \\
    & +\delta^{\frac{(1-p)p^*}{p(p^*-p\sig)}}C^{\frac{1}{\be}}\tilde{C}^{\frac{1}{\be^2}}\left(\fint_{B_r}u^{p\sig}\right)^{\frac1{p\sig}}.
\end{align*}
Observe that $C^{\frac{1}{\be}}\tilde{C}^{\frac{1}{\be^2}}=C(d,s,p)$. Now, choosing $k=0$ yields \eqref{loc.bounded}.
\end{proof}
In the following lemma, we state the Tail estimate for weak solution to \eqref{eqn-scaled}. 

\begin{lemma}\label{Tail}
Assume that $u$ is a weak solution to \eqref{eqn-scaled} satisfying $u \geq 0$ in $B_R(x_0) \subset \Omega$. Let $0<r\leq1$ be such that $B_r(x_0) \subset B_R(x_0)$.
Then
\begin{equation}\label{tailest}
\Tail(u^{+};x_0,r)\leq C(d,p,s)\left( \theta^{\frac{1}{1-p}} \,\underset{B_{r}(x_0)}{\esssup}\,u+\Big(\frac{r}{R}\Big)^\frac{p}{p-1}\Tail(u^{-};x_0,R)\right).
\end{equation}
\end{lemma}

\begin{proof}
 Let $M=\underset{B_r}{\esssup}\,u$ and $\phi\in \C_c^{\infty}(B_r)$ be a cut-off function such that
$0\leq\phi\leq 1$ in $B_r$, $\phi=1$ in $B_{\frac{r}{2}}(x_0)$ and $|\nabla\phi|\leq\frac{8}{r}$ in $B_r$.
By letting $w=u-2M$ and choosing $h=w\phi^p$ as a test function we obtain
\begin{align}\label{tailtest}
\int_{B_r} V(x)u^{p-1}w\phi^p\dx &=\int_{B_r}|\nabla u|^{p-2}\nabla u\cdot\nabla(w\phi^p)\dx\no\\
&\qquad+\theta\int_{B_r}\int_{B_r}\mathcal{A}(u(x,y))(w(x)\phi(x)^p-w(y)\phi(y)^p)\dm\no\\
&\qquad+2\theta\int_{B_r}\int_{\mathbb{R}^d\setminus B_r}\mathcal{A}(u(x,y))w(x)\phi(x)^p\dm\no\\
&=I_1+I_2+I_3.
\end{align}
By Young's inequality,
\begin{align*}
|\nabla w|^{p-2}\nabla w\cdot\nabla(w\phi^p)
&=|\nabla w|^p \phi^p+p\phi^{p-1}w|\nabla w|^{p-2}\nabla w\cdot\nabla\phi\\
&\geq\frac{1}{2}|\nabla w|^p\phi^p-C(p)|u|^p|\nabla\phi|^p
-C(p)M^p|\nabla\phi|^p,
\end{align*}
holds in $B_r$. By the properties of $\phi$, we have
\begin{equation}\label{tailestI1}
I_1=\int_{B_r}|\nabla u|^{p-2}\nabla u\cdot\nabla(w\phi^p) \dx\geq -C(p)M^p r^{-p}|B_r|.
\end{equation}
In view of $(4.11)$ and $(4.9)$ in \cite[Pages 1827--1828]{KuusiHarnack} and using the fact that $r\in(0,1)$, 
there exists $C=C(d,p,s)>0$ such that
\begin{align}\label{tailestI2}
I_2&=\theta\int_{B_r}\int_{B_r}\mathcal{A}(u(x,y))\left(w(x)\phi(x)^p-w(y)\phi(y)^p\right) \dm \geq -C \theta M^p r^{-p}|B_r|,
\end{align}
and
\begin{align}\label{tailestI3}
I_3&=2\theta\int_{B_r}\int_{\rd\setminus B_r}\mathcal{A}(u(x,y))w(x)\phi(x)^p \dm
\geq C \theta M r^{-p}\Tail(u^{+};x_0,r)^{p-1}|B_r|\no\\
&\qquad-C \theta MR^{-p}\Tail(u^{-};x_0,R)^{p-1}|B_r|- C \theta M^p r^{-p}|B_r|. 
\end{align}
Further, 
\begin{align}\label{tailestI4}
        \int_{B_r} |V(x)|u^{p-1}|u-2M|\phi^p\dx \le M^p \int_{B_r} |V(x)| \dx \le M^p \norm{V}_{L^q(B_r)} |B_r|^{\frac{q-1}{q}},
\end{align}
Combining \eqref{tailestI1}-\eqref{tailestI4}, we get from \eqref{tailtest},
\begin{align*}
    \Tail(u^{+};x_0,r)\leq C(d,p,s)\left(M\theta^{\frac{1}{1-p}} + %\theta^{\frac{1}{p-1}} 
    \Big(\frac{r}{R}\Big)^\frac{p}{p-1}\Tail(u^{-};x_0,R) + M\theta^{\frac{1}{1-p}} \left[ \norm{V}_{L^q} r^p |B_r|^{-\frac{1}{q}} \right]^{\frac{1}{p-1}} \right),
\end{align*}
where $r^p|B_r(x_0)|^{\frac{-1}{q}}=(\om_d)^{-\frac{1}{q}} r^{p-\frac dq}\leq C(d)$ using the fact that $r \le 1$ and $q>\frac{d}{p}$. Thus, \eqref{tailest} holds. 
\end{proof}

\section{Harnack inequalities}{\label{harnackinq}}
This section contains the proof of the Harnack inequalities. 
We start with the following lemma that proves the expansion of positivity for a weak supersolution to \eqref{eqn-scaled}.
\begin{lemma}{\label{expan}}
Assume that $u$ is a weak supersolution to \eqref{eqn-scaled} satisfying $u \geq 0$ in $B_R(x_0) \subset \Omega$. Let $0<r\leq1$ be such that $B_r(x_0) \subset B_{\frac R{16}}(x_0)$. Assume $k \geq 0$ and there exists $\tau \in(0,1]$ such that
\begin{align}{\label{expan1}}
    \left|B_r(x_0) \cap\{u \geq k\}\right| \geq \tau\left|B_r(x_0\right)|.
\end{align} Then there exists $\zeta_3(d,p)>0$ such that for $\|V\|_{L^q(\Omega)} \le \zeta_3$,
\begin{align}{\label{expan1.1}}
\underset{B_{4 r}\left(x_0\right)}{{\essinf}}\, u \geq \delta k-\theta^{\frac{1}{p-1}}\left(\frac{r}{R}\right)^{\frac{p}{p-1}} \operatorname{Tail}\left(u^{-} ; x_0, R\right),
\end{align}
where $\delta=\delta(d, p, s, \tau) \in \left(0, \frac{1}{4}\right)$.
\end{lemma}
\begin{proof}
Our proof includes two steps.\\
\textbf{Step 1.} Let $\var>0$. Under the assumption in \eqref{expan1}, we claim that there exists  $C_1=C_1(d, p, s)$ such that
\begin{align}{\label{expan2}}
\left|B_{6r}(x_0) \cap\left\{u \leq 2 \delta k-\frac{1}{2}\theta^{\frac{1}{p-1}}\left(\frac{r}{R}\right)^{\frac{p}{p-1}} \operatorname{Tail}\left(u^{-} ; x_0, R\right)-\var\right\}\right| \leq \frac{C_1}{\tau \log \frac{1}{2 \delta}}\left|B_{6r}(x_0)\right|,
\end{align}
for every $\delta \in\left(0, \frac{1}{4}\right)$.
Consider a  cut-off function $\phi \in \C_c^{\infty}(B_{7 r}(x_0))$ such that $0 \leq \phi \leq 1$ in $B_{7 r}(x_0), \phi=1$ in $B_{6 r}$ and $|\nabla \phi| \leq \frac{8}{r}$ in $B_{7 r}(x_0)$. Define $v:=u+t_\var$, where
$$
t_\var=\frac{1}{2}\theta^{\frac{1}{p-1}}\left(\frac{r}{R}\right)^{\frac{p}{p-1}} \operatorname{Tail}\left(u^{-} ; x_0, R\right)+\var .
$$
As $u$ is a weak supersolution, we choose $v^{1-p} \phi^p$ as a test function and estimate like Lemma \ref{log-estimate} to obtain
\begin{align}{\label{expan3}}
&\int_{B_{6r}}|\nabla \log (u+t_\var)|^p \dx +\theta\iint_{B_{6r} \times B_{6r}}\left|\log \left(\frac{u(x)+t_\var}{u(y)+t_\var}\right)\right|^p \dm  \nonumber\\&\leq C r^d\left(r^{-p}+\theta r^{-p s}+\theta t_\var^{1-p} R^{-p} \operatorname{Tail}\left(u^{-} ; x_0, R\right)^{p-1}\right)\leq Cr^{d-p},
\end{align}
where $C=C(d, p, s)$. For $\delta \in\left(0, \frac{1}{4}\right)$, we denote
$$
w=\left(\min \left\{\log \frac{1}{2 \delta}, \log \frac{k+t_\var}{v}\right\}\right)^{+} .
$$
By \eqref{expan3}, we have
\begin{align}{\label{expan4}}
\int_{B_{6 r}}|\nabla w|^p \dx \leq \int_{B_{6 r}}|\nabla \log v|^p \dx \leq C r^{d-p}.
\end{align}
From \eqref{expan4}, by Hölder's inequality and Poincaré inequality, we obtain
\begin{align}{\label{expan5}}
\int_{B_{6 r}}\left|w-(w)_{B_{6 r}}\right| \dx& \leq Cr^{\frac {d}{p^\prime}}\left(\int_{B_{6 r}}\left|w-(w)_{B_{6 r}}\right|^p \dx\right)^{\frac{1}{p}} \no\\&\leq Cr^{1+\frac{d}{p^{\prime}}}\left(\int_{B_{6r}}|\nabla w|^p \dx\right)^{\frac{1}{p}} \leq C\left|B_{6 r}\right|,
\end{align}
where $p^{\prime}=\frac{p}{p-1}$ and $(w)_{B_{6 r}}=\fint_{B_{6r}} v \dx$. We observe that $\{w=0\}=\left\{v \geq k+t_\var\right\}=\{u \geq k\}$. By the assumption \eqref{expan1}, it follows that
\begin{align}{\label{expan6}}
    \left|B_{6 r} \cap\{w=0\}\right| \geq \frac{\tau}{6^d}\left|B_{6r}\right|.
\end{align}
Following the proof of \cite[Lemma 3.1]{KuusiHarnack} and using \eqref{expan6}, we obtain
\begin{align}{\label{expan7}}
\log \frac{1}{2 \delta} & =\frac{1}{\left|B_{6 r} \cap\{w=0\}\right|} \int_{B_{6 r} \cap\{w=0\}}\left(\log \frac{1}{2 \delta}-w(x)\right) \dx  \leq \frac{6^d}{\tau}\left(\log \frac{1}{2 \delta}-(w)_{B_{6 r}}\right).
\end{align}
Now integrating \eqref{expan7} over the set $B_{6 r} \cap\left\{w=\log \frac{1}{2 \delta}\right\}$ and using \eqref{expan5}, there exists $C_1=C_1(d, p, s)$ such that
$$
\left|\left\{w=\log \frac{1}{2 \delta}\right\} \cap B_{6 r}\right| \log \frac{1}{2 \delta} \leq \frac{6^d}{\tau} \int_{B_{6 r}}\left|w-(w)_{B_{6 r}}\right| \dx \leq \frac{C_1}{\tau}\left|B_{6 r}\right| .
$$
Hence, for any $\delta \in\left(0, \frac{1}{4}\right)$, we have
$$
\left|B_{6 r} \cap\left\{v \leq 2 \delta\left(k+t_\var\right)\right\}\right| \leq \frac{C_1}{\tau} \frac{1}{\log \frac{1}{2 \delta}}\left|B_{6r}\right| .
$$
This implies \eqref{expan2}.
\smallskip\\\textbf{Step 2.} We claim that, for every $\var>0$, there exists a constant $\delta=\delta(d, p, s,  \tau) \in\left(0, \frac{1}{4}\right)$ such that
\begin{align}{\label{expan8}}
\underset{B_{4 r}(x_0)}{\essinf}\, u \geq \delta k-\theta^{\frac{1}{p-1}}\left(\frac{r}{R}\right)^{\frac{p}{p-1}} \operatorname{Tail}\left(u^{-} ; x_0, R\right)-2 \var .
\end{align}As a consequence of \eqref{expan8}, the estimate \eqref{expan1.1} follows.
To prove \eqref{expan8}, without loss of generality, we may assume that
\begin{align}{\label{expan9}}
    \delta k \geq\theta^{\frac{1}{p-1}}\left(\frac{r}{R}\right)^{\frac{p}{p-1}} \operatorname{Tail}\left(u^{-} ; x_0, R\right)+2 \var.
\end{align}
Otherwise \eqref{expan8} holds true, since $u \geq 0$ in $B_r$.
Let $\rho \in[r, 6 r]$ and $\phi \in \C_c^{\infty}(B_\rho(x_0))$ be a cut-off function such that $0 \leq \phi \leq 1$ in $B_\rho(x_0)$. For any $l \in(\delta k, 2 \delta k)$, we denote $w=(l-u)^{+}$. From Lemma \ref{energy estimate}, for $C=C(d, p, s)$, we obtain
 \begin{align}{\label{expan10}}
       \fint_{B_{\rho}} \phi^{p} |\nabla w|^{p} \dx 
       %+ \theta\iint_{B_{\rho} \times B_{\rho}} |w(x)\phi(x) - w(y)\phi(y)|^{p} \dm
       &\le C \Bigg( \fint_{B_{\rho}} w^{p} |\nabla \phi|^{p} \dx + \fint_{B_{\rho}}| V(x)| u^{p-1} w \phi^p \dx  \nonumber \\
       &+ \theta\fint_{B_\rho}\int_{B_\rho} \max\{w(x), w(y)\}^{p} |\phi(x)-\phi(y)|^{p} \dm   \nonumber \\
       &+ \theta\underset{x \in \text{supp} (\phi)}{\esssup}\, \int_{\rd \setminus B_{\rho}} \frac{w(y)^{p-1}}{|x-y|^{d+ps}} \dy \cdot \fint_{B_{\rho}} w\phi^{p} \dx\Bigg),
   \end{align}
We will apply Lemma \ref{iteration} to conclude the proof. For $j=0,1,2, \ldots$, we denote
\begin{align*}
l=k_j=\delta k+2^{-j-1} \delta k, \quad \rho=\rho_j=4 r+2^{1-j} r, \quad \hat{\rho}_j=\frac{\rho_j+\rho_{j+1}}{2}.
\end{align*}
Then $l \in(\delta k, 2 \delta k), \rho_j, \hat{\rho_j} \in(4 r, 6 r)$ and
\begin{align}{\label{obser}}
k_j-k_{j+1}=2^{-j-2} \delta k \geq 2^{-j-3} k_j,
\end{align}
for every $j$. Set $B_j=B_{\rho_j}\left(x_0\right), \hat{B}_j=B_{\hat{\rho}_j}\left(x_0\right)$ and we observe that
\begin{align*}
w_j=\left(k_j-u\right)^{+} \geq 2^{-j-3} k_j \chi_{\left\{u<k_{j+1}\right\}}.
\end{align*}
We now consider a sequence of cut-off functions $\left(\phi_j\right)_{j=0}^{\infty} \subset \C_c^{\infty}(\hat{B}_j)$ such that $0 \leq \phi_j \leq 1$ in $\hat{B}_j, \phi_j=1$ in $B_{j+1}$ and $\left|\nabla \phi_j\right| \leq \frac{2^{j+3}}{r}$. We choose $\phi=\phi_j, w=w_j$ in \eqref{expan10} to have
 \begin{align}{\label{expan101}}
       \fint_{B_{j}} \phi_j^{p} |\nabla w_j|^{p} \dx 
       %+ \theta\iint_{B_{\rho} \times B_{\rho}} |w(x)\phi(x) - w(y)\phi(y)|^{p} \dm
       &\le C \Bigg( \fint_{B_{j}} w_j^{p} |\nabla \phi_j|^{p} \dx + \fint_{B_{j}}| V(x)| u^{p-1} w _j\phi_j^p \dx  \nonumber \\
       &+ \theta\fint_{B_j}\int_{B_j} \max\{w_j(x), w_j(y)\}^{p} |\phi_j(x)-\phi_j(y)|^{p} \dm   \nonumber \\
       &+ \theta\underset{x \in \text{supp} (\phi_j)}{\esssup}\, \int_{\rd \setminus B_{j}} \frac{w_j(y)^{p-1}}{|x-y|^{d+ps}} \dy \cdot \fint_{B_{j}} w_j\phi_j^{p} \dx\Bigg)=:J_1+J_2+J_3+J_4.
   \end{align}
We estimate each term present in the right-hand side of \eqref{expan101}. Using the properties of $\phi_j$, we have
\begin{align}{\label{J_1ex}}
J_1=\fint_{B_{j}} w_j^{p} |\nabla \phi_j|^{p} \dx\leq C 2^{jp}r^{-p}\fint_{B_{j}} w_j^{p}  \dx.
    \end{align}
 For the term $J_2$, recall that $V\in L^q(\Omega)$, where $q$ is described in \eqref{weight}, and estimate as   
    \begin{align}{\label{expan13}}
\int_{B_{j}}| V(x)| u^{p-1} w_j \phi_j^p \dx=\int_{B_{j}}| V(x)| u^{p-1} (k_j-u)^+ \phi_j^p \dx &\leq k_j^{p-1}\int_{B_{j}}| V(x)| (k_j-u)^+ \phi_j^p \dx\no\\&\leq {k}_j^{p-1} \left\|V \phi_j^{p-1} \right\|_{L^{q}(B_j)} \left\|{w}_j\phi_j\right\|_{L^{\frac{q}{q-1}}(B_j)} \no\\&\leq {k}_j^{p-1} \left\|V  \right\|_{L^{q}(B_j)} \left\|{w}_j\phi_j\right\|_{L^{\frac{q}{q-1}}(B_j)}.
\end{align}
% Applying the generalized H\"{o}lder's inequality $\frac{1}{p_1} = \frac{\al}{p_2} + \frac{1-\al}{p_3}$ with the conjugate triplet $(p_1,p_2,p_3)$, where 
% \begin{align*}
%     p_1=\frac{q}{q-1}, p_2= p^*, p_3=1, \text{ and } \al=\frac{p^*}{(p^*-1)q},
% \end{align*}
% and observing the fact that 
% \begin{align*}
%     q>\frac{d}{p} \Longrightarrow \frac{q}{q-1} < \frac{d}{d-p} <p^* \Longrightarrow q > \frac{p^*}{p^*-1} \Longrightarrow \al < 1,
% \end{align*}
% we now estimate
% \begin{align*}
%     \left\|{w}_j\phi_j \right\|_{L^{\frac{q}{q-1}}(B_j)}  \le \left\| {w}_j \phi_j \right\|_{L^{p^*}(B_j)}^{\al} \left\|{w}_j\phi_j\right\|_{L^1(B_j)}^{1-\al}.
% \end{align*}
%Now, applying the Young's inequality with coefficients $(\frac{p}{\al},\frac{p}{p-\al})$, 
Now, proceeding as in Proposition \ref{local-boundedness-I} we obtain 
\begin{align*}
{k}_j^{p-1} \left\|{w}_j\phi_j\right\|_{L^{\frac{q}{q-1}}(B_j)} \leq \|{w}_j\phi_j\|_{L^{p^*}(B_j)}^p+{k}_j^{\frac{p(p-1)}{p-\al}}\|{w}_j\phi_j\|_{L^{1}(B_j)}^{\frac{p(1-\al)}{p-\al}},
\end{align*}
where $\al < 1$.
By $W^{1,p}_0(B_j) \hookrightarrow L^{p^*}(B_j)$ and the fact that ${w}_j\phi_j \in W^{1,p}_0(B_j)$, we get 
\begin{align*}
    \|{w}_j\phi_j\|_{L^{p^*}(B_j)}^p \le C(d,p) \left( \int_{B_j}\phi_j^p|\nabla {w}_j|^p \dx + \int_{B_j}{w}_j^p|\nabla\phi_j|^p\dx \right). 
\end{align*}
Observe that $w_j$ is supported in $\{ u<k_j\}$ and $w_j \le k_j$ a.e. in $B_j \cap \{ u<k_j\}$. Therefore, choosing $\zeta_3 =  \zeta_3(d,p) >0$ small enough so that $\|V\|_{L^q(\Omega)}C(d,p)\leq \frac 12$, and noting that $\rho_j\in(4r,6r)$, we use the properties of $\phi_j$ and \eqref{expan13}, to get
\begin{align}{\label{J_2ex}}
J_2&=\fint_{B_{j}}| V(x)| u^{p-1} w_j \phi_j^p \dx \no\\&\leq \frac 12\fint_{B_j}\phi_j^p|\nabla {w}_j|^p \dx +\fint_{B_{j}} w_j^{p} |\nabla \phi_j|^{p} \dx+C r^{-d}
\|V\|_{L^q(\Omega)}{k}_j^{\frac{p(p-1)}{p-\al}}\|{w}_j\phi_j\|_{L^{1}(B_j)}^{\frac{p(1-\al)}{p-\al}}\no\\&\leq \frac 12\fint_{B_j}\phi_j^p|\nabla {w}_j|^p \dx +\fint_{B_{j}} w_j^{p} |\nabla \phi_j|^{p} \dx+C r^{-d}\|V\|_{L^q(\Omega)}{k}_j^{\frac{p(p-1)}{p-\al}+\frac{p(1-\alpha)}{p-\alpha}}\|\phi_j\|_{L^{1}(B_j \cap\left\{u<k_j\right\})}^{\frac{p(1-\al)}{p-\al}}\no\\&\leq  \frac 12\fint_{B_j}\phi_j^p|\nabla {w}_j|^p \dx +\fint_{B_{j}} w_j^{p} |\nabla \phi_j|^{p} \dx+C r^{-d}\|V\|_{L^q(\Omega)}k_j^p\left|B_j \cap\left\{u<k_j\right\}\right|^{\frac 1\sigma}\no\\&\leq \frac 12\fint_{B_j}\phi_j^p|\nabla {w}_j|^p \dx +C 2^{jp}r^{-p}\fint_{B_{j}} w_j^{p}+C r^{\frac d\sigma-d}\|V\|_{L^q(\Omega)}k_j^p\left(\frac{\left|B_j \cap\left\{u<k_j\right\}\right|}{|B_j|}\right)^{\frac 1\sigma},
\end{align}
where $\sig=\frac{(dp-d+p)q-d}{(dp-d+p)q-dp}>1$ and $C=C(d,p,s)$. In the second last line of \eqref{J_2ex} we have used the fact $ \frac{p(p-1)}{p- \alpha} = p - \frac{1}{\sigma}$. 
Further, we estimate
\begin{align}{\label{J_3ex}}
    J_3=\theta\fint_{B_j}\int_{B_j} \max\{w_j(x), w_j(y)\}^{p} |\phi_j(x)-\phi_j(y)|^{p} \dm   & \le C 2^{jp} r^{-p} \fint_{B_j} w_j^p(x) \left( \int_{B_j} \frac{\dy}{\abs{x-y}^{d+ps-p}} \right) \dx \no\\
    & \le \frac{C2^{jp}r^{-p}r^{p-ps}}{p(1-s)} \fint_{B_j} w_j^p \dx\no\\& \le C2^{jp}r^{-p}\fint_{B_j} w_j^p \dx, 
\end{align}
where $C=C(p,s)$. To estimate $J_4$, we observe that, for any $x \in \operatorname{supp} \phi_j \subset \hat{B}_j$ and $y \in \mathbb{R}^d \backslash B_j$, we have
$$
\frac{\left|y-x_0\right|}{|y-x|}=\frac{\left|y-x+x-x_0\right|}{|y-x|} \leq 1+\frac{\left|x-x_0\right|}{|y-x|} \leq 1+\frac{\hat{\rho}_j}{\rho_j-\hat{\rho}_j}\le2^{j+4} .
$$Thus, we obtain
\begin{align}{\label{J_4ex}}
   J_4&= \theta\underset{x \in \text{supp} (\phi_j)}{\esssup}\,
\int_{\rd \setminus B_{j}} \frac{w_j(y)^{p-1}}{|x-y|^{d+ps}} \dy \cdot \fint_{B_{j}} w_j\phi_j^{p} \dx\no\\
    &\le \theta C 2^{j(d+sp)}\int_{\rd \setminus B_{j}} \frac{w_j(y)^{p-1}}{|x_0-y|^{d+ps}} \dy \cdot \fint_{B_{j}} w_j\phi_j^{p} \dx\no\\
    &\le\theta C 2^{j(d+sp)}\left(\int_{\rd \setminus B_{j}} \frac{k_j^{p-1}}{|x_0-y|^{d+ps}} \dy +\int_{\rd \setminus B_{R}} \frac{(u(y)^-)^{p-1}}{|x_0-y|^{d+ps}} \dy \right)\cdot \fint_{B_{j}} w_j\dx\no\\
    &\le C 2^{j(d+sp)}\left(k_j^{p-1} r^{-ps}+r^{-p}\theta\left(\frac{r}{R}\right)^p \operatorname{Tail}\left(u^{-} ; x_0, R\right)^{p-1}\right) \cdot \fint_{B_{j}} w_j\dx\no\\
    &\le C 2^{j(d+sp)}k_j^{p-1} r^{-p}\fint_{B_{j}} w_j\dx,
\end{align} 
where $C=C(d,s,p)$. Here we have used the fact that $r \in(0,1]$ along with \eqref{expan9}, $\delta k<k_j$ and $u \geq 0$ a.e. in $B_R$.
Also, observe that 
\begin{align}{\label{sigmap}}\fint_{B_j}w_j^p\dx\leq \left(\fint_{B_j}w_j^{p\sig}\dx\right)^{\frac1\sig}.
\end{align}
Merging \eqref{J_1ex}, \eqref{J_2ex}, \eqref{J_3ex} and \eqref{J_4ex}, we obtain from \eqref{expan101} using \eqref{sigmap},
\begin{align}{\label{expan102}}
       \fint_{B_{j}} \phi_j^{p} |\nabla w_j|^{p} \dx 
       %+ \theta\iint_{B_{\rho} \times B_{\rho}} |w(x)\phi(x) - w(y)\phi(y)|^{p} \dm
       \le&\, C 2^{jp}r^{-p}\left(\fint_{B_j}w_j^{p\sig}\dx\right)^{\frac1\sig}+C r^{\frac d\sigma-d}k_j^p\left(\frac{\left|B_j \cap\left\{u<k_j\right\}\right|}{|B_j|}\right)^{\frac 1\sigma}\no\\&+C2^{j(d+sp)}k_j^{p-1} r^{-p}\fint_{B_{j}} w_j\dx\no\\ \leq &\,C 2^{jp}r^{-p}k_j^p\left(\frac{\left|B_j \cap\left\{u<k_j\right\}\right|}{|B_j|}\right)^{\frac 1\sigma}+C r^{\frac d\sigma-d}k_j^p\left(\frac{\left|B_j \cap\left\{u<k_j\right\}\right|}{|B_j|}\right)^{\frac 1\sigma}\no\\&+C2^{j(d+sp)}k_j^{p-1} r^{-p}\left(\fint_{B_{j}} w_j^\sigma\dx\right)^{\frac 1\sigma}\no\\\leq &\,C 2^{j{(p+sp+d)}}r^{-p}k_j^p\left(\frac{\left|B_j \cap\left\{u<k_j\right\}\right|}{|B_j|}\right)^{\frac 1\sigma}+C r^{\frac d\sigma-d}k_j^p\left(\frac{\left|B_j \cap\left\{u<k_j\right\}\right|}{|B_j|}\right)^{\frac 1\sigma}.
   \end{align}
Here we again use the fact that $w_j\le k_j$ a.e. in $B_j \cap \{ u<k_j \}$. 
By applying the Sobolev inequality as in \eqref{e.friedrich} along with \eqref{J_1ex}, \eqref{sigmap}, \eqref{expan102} and the properties of $\phi_j$, there exists $C=C(d, p, s,\|V\|)$ such that
\begin{align}{\label{expan18}}
(k_{j}-{k}_{j+1})^{p}\left(\frac{\left|B_{j+1} \cap\left\{u<k_{j+1}\right\}\right|}{|B_{j+1}|}\right)^{\frac{p}{p^*}}& \leq\left(\fint_{B_{j+1}}({w}_j\phi_j)^{p^*}\dx\right)^{\frac{p}{p^*}} \leq C\left(\fint_{B_{j}}({w}_j\phi_j)^{p^*}\dx\right)^{\frac{p}{p^*}}\no\\
& \leq C r^p \fint_{B_j}|\nabla(w_j \phi_j)|^p \dx \no\\
& \leq C 2^{j(d+p s+p)} \left(1+r^{p-d+\frac d\sigma}\right)k_j^p \left(\frac{\left|B_j \cap\left\{u<k_j\right\}\right|}{|B_j|}\right)^{\frac 1\sigma} .
\end{align}
As $r\in(0,1)$ and $p-d+\frac d\sigma\geq 0$, we have from \eqref{expan18} that
\begin{align}{\label{expan20}}
   (k_{j}-{k}_{j+1})^{p}\left(\frac{\left|B_{j+1} \cap\left\{u<k_{j+1}\right\}\right|}{|B_{j+1}|}\right)^{\frac{p}{p^*}}\leq C 2^{j(d+p s+p)}k_j^p \left(\frac{\left|B_j \cap\left\{u<k_j\right\}\right|}{|B_j|}\right)^{\frac 1\sigma} .
\end{align}Let
$$
Y_j=\left(\frac{\left|B_j \cap\left\{u<k_j\right\}\right|}{|B_j|}\right)^{\frac{1}{p\sigma}}.
$$
From \eqref{expan20} and \eqref{obser} there exists $C_2=C_2(d, p, s,\|V\|)>1$ such that
\begin{align}{\label{expan21}}
&\left(Y_{j+1}\right)^{\frac{p^2\sig}{p^*}} \leq C_2 2^{j(d+2 p+p s)} Y_j^p \implies Y_{j+1}\leq C_2\hat{C}^j\left(Y_j\right)^{1+\beta},
\end{align}
where using $p\sig <p^*$, we see that $\beta:=\frac{p^*}{p\sig}-1>0$ and $\hat{C}:=2^{\left(\frac{d+2p+ps}{p}\right)\frac{p^*}{p\sig}}>1.$ We choose $c_0=C_2$ and $b=\hat{C}$ in Lemma \ref{iteration}. By \eqref{expan9}, 
\begin{align*}
k_0=\frac{3}{2} \delta k \leq 2 \delta k-\frac{1}{2}\theta^{\frac{1}{p-1}}\left(\frac{r}{R}\right)^{\frac{p}{p-1}} \operatorname{Tail}\left(u^{-} ; x_0, R\right)-\var .
\end{align*}
By \eqref{expan2},
\begin{align}{\label{expan22}}
Y_0 \leq \Bigg(\frac{\left|B_{6r} \cap\left\{u \leq 2 \delta k-\frac 12\theta^{\frac{1}{p-1}}\left(\frac{r}{R}\right)^{\frac{p}{p-1}} \operatorname{Tail}\left(u^{-} ; x_0, R\right)-\var\right\}\right|}{\left|B_{6r}\right|}\Bigg)^{\frac{1}{p\sig}} \leq \left(\frac{C_1}{\tau \log \frac{1}{2 \delta}}\right)^{\frac{1}{p\sig}},
\end{align}
for $C_1=C_1(d, p, s)$ and for every $\delta \in\left(0, \frac{1}{4}\right)$. Using \eqref{expan22} we now choose $\delta=\delta(d, p, s, \|V\|, \tau)$ as
\begin{align*}
0<\delta:=\frac{1}{4} \exp \left(-\frac{C_1 c_0^{\frac{p\sig}{\beta}} b^{\frac{p\sig}{\beta^2}}}{\tau}\right)<\frac{1}{4},
\end{align*}
so that the estimate $Y_0 \leq c_0^{-\frac{1}{\beta}} b^{-\frac{1}{\beta^2}}$ holds in Lemma \ref{iteration}. Therefore, in \eqref{expan21}, Lemma \ref{iteration} infers that $Y_j \rightarrow 0$ as $j \rightarrow \infty$. Thus we have
\begin{align*}
\underset{B_{4 r}(x_0)}{\essinf } \,u \geq \delta k,
\end{align*}
which gives \eqref{expan8} and so \eqref{expan1.1} holds. This completes the proof. 
\end{proof}

Proceeding similarly as in the proof of \cite[Lemma 4.1]{KuusiHarnack}, along with an application of Lemma \ref{expan}, we obtain the following preliminary version of the weak Harnack inequality, compared to Theorem \ref{weakhar_intro}.

\begin{lemma}{\label{lower}} 
Assume that $u$ is a weak supersolution to \eqref{eqn-scaled} satisfying $u \geq 0$ in $B_R(x_0) \subset \Omega$. Let $0<r\leq1$ be such that $B_r(x_0) \subset B_{\frac R{2}}(x_0)$. Then there exists $\zeta_3(d,p)>0$ such that for $\|V\|_{L^q(\Omega)} \le \zeta_3$, the following holds
$$
\left(\fint_{B_r(x_0)} u^\la \dx\right)^{\frac{1}{\la}} \leq C\, \underset{B_r(x_0)}{\essinf}\, u+C\theta^{\frac{1}{p-1}}\left(\frac{r}{R}\right)^{\frac{p}{p-1}} \operatorname{Tail}\left(u^{-} ; x_0, R\right),
$$
where $\la=\la(d, p, s) \in(0,1)$ and $C=C(d, p, s) \geq 1$.
\end{lemma}

From the local boundedness and the Tail estimate, we now obtain another weak Harnack inequality. For that, we require the following iteration lemma from \cite[Lemma 1.1]{acta}.  
\begin{lemma}\label{iteration1}
Let  $0\leq T_0\leq \tilde{t} \leq T_1$ and assume that $f:[T_0,T_1]\to[0,\infty)$ is a nonnegative bounded function. Suppose that for $T_0\leq \tilde{t} <\hat{t} \leq T_1$, we have
\begin{equation*}\label{itt}
f(\tilde{t})\leq A(\hat{t}-\tilde{t})^{-\alpha} +B +\theta f(\hat{t}),
\end{equation*}
where $A,B,\alpha,\theta$ are nonnegative constants and $\theta<1$. 
Then there exists $C=C(\alpha,\theta)$ such that for every $\rho,R$ and $T_0\leq\rho<R\leq T_1$, we have
\begin{equation*}\label{itt1}
f(\rho)\leq C(A(R-\rho)^{-\alpha}+B).
\end{equation*}
\end{lemma}

\begin{proposition}{\label{prop}}
Assume that $u$ is a weak supersolution to \eqref{eqn-scaled} satisfying $u \geq 0$ in $B_R(x_0) \subset \Omega$. Let $0<r\leq1$ be such that $B_r(x_0) \subset B_{R}(x_0)$.
Then there exists $\zeta_2(d,p)>0$ such that for $\|V\|_{L^q(\Omega)} \le \zeta_2$,
\begin{align}\label{mixed-harnack-1}
    \underset{B_{\frac{r}{2}}(x_0)}{\esssup}\,\,u \le C(d,s,p) \left( \theta^{\frac{1}{p-1}}\Big(\frac{r}{R}\Big)^\frac{p}{p-1}\Tail(u^{-};x_0,R) + \left(\fint_{B_r(x_0)}u^{t}\right)^{\frac1{t}} \right),
\end{align}
where $t \in (0, p\sigma)$ with $\sig=\frac{(dp-d+p)q-d}{(dp-d+p)q-dp}$.
\end{proposition}

\begin{proof}
   Let $0<\rho<r$. As $\theta\in(0,1]$, using \eqref{loc.bounded} of Proposition \ref{local-boundedness-I}  and \eqref{tailest}, we get 
    \begin{align*}
         \underset{B_{\frac{\rho}{2}(x_0)}}{\esssup}\,u & \leq C \delta \theta^{\frac{1}{p-1}}\left( \theta^{\frac{1}{1-p}}\,\underset{B_\rho(x_0)}{\esssup}\,u+\Big(\frac{\rho}{R}\Big)^\frac{p}{p-1}\Tail(u^{-};x_0,R) \right) +C\delta^{\frac{(1-p)p^*}{p(p^*-p\sig)}}\left(\fint_{B_\rho(x_0)}u^{p\sig}\right)^{\frac1{p\sig}}, \\
         & \le C(d,s,p) \delta \left(\underset{B_\rho(x_0)}{\esssup}\,u+\theta^{\frac{1}{p-1}}\Big(\frac{\rho}{R}\Big)^\frac{p}{p-1}\Tail(u^{-};x_0,R) \right) +C\delta^{\frac{(1-p)p^*}{p(p^*-p\sig)}}\left(\fint_{B_\rho(x_0)}u^{p\sig}\right)^{\frac1{p\sig}},
    \end{align*}
   where $C=C(d,s,p)$. We set $\rho=(\eta-\eta')r,$ with $\frac12\leq\eta'<\eta\leq1$. We have by a covering argument that
   \begin{align*}
       \underset{B_{\eta' r}(x_0)}{\esssup}\,u &\le  C \left( \frac{\delta^{\frac{(1-p)p^*}{p(p^*-p\sig)}}}{(\eta - \eta')^{\frac{d}{p \sig}}} \left( \fint_{B_{\eta r}(x_0)} u^{p \sigma} \right)^{\frac{1}{p \sigma}} + \delta \,\underset{B_{\eta r}(x_0)}{\esssup}\,u + \delta \theta^{\frac{1}{p-1}}\Big(\frac{r}{R}\Big)^\frac{p}{p-1}\Tail(u^{-};x_0,R)
      \right),
   \end{align*}
   where for $t \in (0, p \sigma)$,
   \begin{align*}
       \left( \fint_{B_{\eta r}(x_0)} u^{p \sigma} \right)^{\frac{1}{p \sigma}} \le \left( \underset{B_{\eta r}(x_0)}{\esssup}\,u \right)^{\frac{p \sigma -t}{p \sigma}} \left( \fint_{B_{\eta r}(x_0)} u^{t} \right)^{\frac{1}{p \sigma}}.
   \end{align*}
Choosing $\delta=\frac1{4C(d,p,s)}$ and applying Young's inequality with the conjugate pair $(\frac{p \sigma}{p \sigma -t}, \frac{p \sigma}{t})$, we get
\begin{align*}
    \underset{B_{\eta'r}(x_0)}{\esssup}\,u\leq \frac12\,\underset{B_{\eta r}(x_0)}{\esssup}\,u+\frac{C}{(\eta-\eta')^{\frac{d}{t}}}\left(\fint_{B_{\eta r}(x_0)}u^{t}\right)^{\frac1{t}} + C\theta^{\frac{1}{p-1}} \Big(\frac{r}{R}\Big)^\frac{p}{p-1}\Tail(u^{-};x_0,R),
\end{align*}
for some $C=C(d,p,s)$. Now applying Lemma \ref{iteration1}, with $f(z) = \underset{B_{zr}(x_0)}{\esssup}\,u, \hat{t} = \eta , \tilde{t} = \eta', \alpha = \frac{d}{t}$, we obtain 
\begin{align*}
    \underset{B_{\frac{r}{2}}(x_0)}{\esssup}\,u \le C(d,s,p) \left( \theta^{\frac{1}{p-1}}\Big(\frac{r}{R}\Big)^\frac{p}{p-1}\Tail(u^{-};x_0,R) + \left(\fint_{B_{r}(x_0)}u^{t}\right)^{\frac1{t}} \right),
\end{align*}
for every $t \in (0, p\sigma)$ with $\sig=\frac{(dp-d+p)q-d}{(dp-d+p)q-dp}$. This is indeed \eqref{mixed-harnack-1}.
\end{proof} 

\begin{theorem}{\label{harnack}}
Assume that $u$ is a weak solution to \eqref{eqn-scaled} satisfying $u \geq 0$ in $B_R(x_0) \subset \Omega$. Let $0<r\leq1$ be such that $B_r(x_0) \subset B_{\frac R{2}}(x_0)$. Then there exists $\zeta(d,p)>0$ such that for $\|V\|_{L^q(\Omega)} \le \zeta$,
\begin{equation*}
    \underset{B_{\frac{r}{2}}(x_0)}{\esssup}\,u \le C \,\underset{B_r(x_0)}{\essinf}\, u + C\theta^{\frac{1}{p-1}}\left(\frac{r}{R}\right)^{\frac{p}{p-1}} \Tail_{p-1,sp,p}(u^-; x_0, R),
\end{equation*}
where $C=C(d,p,s)$.
\end{theorem} 
\begin{proof}
We take $\zeta = \min \{\zeta_2, \zeta_3\}>0$. Since $p\sigma>1$, we choose $t=\la$ in Proposition \ref{prop}, and then combine with Lemma \ref{lower}, to conclude the proof.   
\end{proof}

\begin{theorem}{\label{weakharnack}} 
Assume that $u$ is a weak supersolution to \eqref{eqn-scaled} satisfying $u \geq 0$ in $B_R(x_0) \subset \Omega$. Let $0<r\leq1$ be such that $B_r(x_0) \subset B_{\frac R{2}}(x_0)$. Then there exists $\hat{\zeta}(m,d,p)>0$ such that for $\|V\|_{L^q(\Omega)} \le \hat{\zeta}$,
$$
\left(\fint_{B_{\frac r2}\left(x_0\right)} u^l \dx\right)^{\frac{1}{l}} \leq C \,\underset{B_r(x_0)}{\essinf}\, u+C\theta^{\frac{1}{p-1}}\left(\frac{r}{R}\right)^{\frac{p}{p-1}} \Tail_{p-1,sp,p}\left(u^{-} ; x_0, R\right),
$$
whenever $0<l<\kappa(p-1)$, with $\kappa$ as given in \eqref{kappa}. Here $C=C(d, p, s)$.
\end{theorem} 
\begin{proof}
Let $r \in(0,1), \frac{1}{2}<\tau^{\prime}<\tau \leq \frac{3}{4}$ and $\phi \in \C_c^{\infty}(B_{\tau r})$ be such that $0 \leq \phi \leq 1$ in $B_{\tau r}, \phi=1$ in $B_{\tau^{\prime} r}$ and $|\nabla \phi| \leq \frac{4}{\left(\tau-\tau^{\prime}\right) r}$. For $t>0$ and $m\in(1, p)$, we set
$$
h=u+t \; \text { and } \; w=(u+t)^{\frac{p-m}{p}} .
$$
Observe that
\begin{equation}{\label{weakhar1}}
\int_{B_r} w^p|\nabla \phi|^p \dx \leq \frac{C(p) r^{-p}}{\left(\tau-\tau^{\prime}\right)^p} \int_{B_{\tau r}} w^p \dx,    
\end{equation}
Noting $r \in(0,1]$ and using the fact that support of $\nabla\phi$ is a subset of $B_{\tau r}$, we have
\begin{align}{\label{weakhar2}} 
\theta\iint_{B_r \times B_r} \max \{w(x), w(y)\}^p|\phi(x)-\phi(y)|^p \dm &\leq \frac{C(p)r^{-p}}{\left(\tau-\tau^{\prime}\right)^p} \iint_{B_{\tau r} \times B_{\tau r}} \frac{(w^p(x)+ w^p(y))}{|x-y|^{d+ps-p}} \dx  \nonumber\\&\leq r^{p-ps}\frac{C(p) r^{-p}}{\left(\tau-\tau^{\prime}\right)^p} \int_{B_{\tau r}} w^p \dx\nonumber\\&\leq \frac{C(p) r^{-p}}{\left(\tau-\tau^{\prime}\right)^p} \int_{B_{\tau r}} w^p \dx,   
\end{align}
for some $C=C(d,p,s)$.
Assume that $\operatorname{Tail}\left(u^{-} ; x_0, R\right)$ is positive. Then for any $\varepsilon>0$ and $r \in(0,1)$ choosing
$$
t=\frac{1}{2}\theta^{\frac{1}{p-1}}\left(\frac{r}{R}\right)^{\frac{p}{p-1}} \operatorname{Tail}\left(u^{-} ; x_0, R\right)+\varepsilon>0,
$$
and noting that
\begin{equation}{\label{weakhar3}}
\underset{z \in \operatorname{supp} \phi}{\esssup}\, \int_{\rd \backslash B_r} \frac{\dy}{|z-y|^{d+ps}} \leq C(d, p, s) r^{-p},
\end{equation}
we obtain
\begin{align}{\label{weakhar4}}
&\left(\theta\underset{z \in \operatorname{supp} \phi}{\esssup}\, \int_{\rd \backslash B_r}\frac{\dy}{|z-y|^{d+ps}} +\theta t^{1-p} R^{-p} \operatorname{Tail}(u^{-} ; x_0, R)^{p-1}\right)\int_{B_r} w^p \phi^p \dx  \no\\&\leq C(d, p, s) \frac{r^{-p}}{\left(\tau-\tau^{\prime}\right)^p} \int_{B_{\tau r}} w^p \dx,
\end{align}
where we use the fact that $\tau-\tau'\in (0,1)$. If $\operatorname{Tail}\left(u^{-} ; x_0, R\right)=0$, we can choose $t=\varepsilon>0$ and again using \eqref{weakhar3} the estimate in \eqref{weakhar4} follows. We take $\hat{\zeta} = \min \{\zeta_1, \zeta_3\}>0$. Now using Sobolev inequality in \eqref{e.friedrich} and the fact that $\phi = 1$ in $B_{\tau^{\prime} r}, r \in(0,1)$, we combine Lemma \ref{3.3} with the estimates \eqref{weakhar1}, \eqref{weakhar2}, and \eqref{weakhar4}. Consequently, we get
\begin{align*}
\left(\fint_{B_{\tau^{\prime} r}} h^{\frac{d(p-m)}{d-p}} \dx\right)^{\frac{p}{p^*}} & =\left(\fint_{B_{\tau^{\prime} r}} w^{p^*} \dx\right)^{\frac{p}{p^*}} \leq\left(\fint_{B_{\tau r}}|w \phi|^{p^*} \dx\right)^{\frac{p}{p^*}} \no\\
& \leq(\tau r)^{p-d} \int_{B_{\tau r}}|\nabla(w \phi)|^p \dx \leq \frac{C}{\left(\tau-\tau^{\prime}\right)^p} \fint_{B_{\tau r}} w^p \dx,
\end{align*}
where $C=C(d, p, s, m)$. For $m \in(1, p)$ using the Moser iteration technique as in \cite[Theorem 8.18]{gilbarg1998elliptic} and \cite[Theorem 1.2]{Trudinger}, we get
\begin{align*}
\left(\fint_{B_{\frac{r}{2}}} h^l \dx\right)^{\frac{1}{l}} &\leq C\left(\fint_{B_{\frac{3 r}{4}}} h^{l^{\prime}} \dx\right)^{\frac{1}{l^{\prime}}},\; \forall\, 0<l^{\prime}<l<\frac{d(p-1)}{d-p}.
\end{align*}
Also observe that $$\fint_{B_{\frac{r}{2}}} u^l\leq\fint_{B_{\frac{r}{2}}} h^l,\; \forall\, l>0.$$
For $\la \in(0,1)$ as in Lemma \ref{lower}, we take $l^{\prime}=\la$ to obtain
\begin{align}{\label{weakhar7}}
&{\left(\fint_{B_{\frac{r}{2}}} u^l \dx\right)^{\frac{1}{l}}} \leq C\left(\fint_{B_{\frac{3 r}{4}}} h^{l^{\prime}} \dx\right)^{\frac{1}{l^{\prime}}} \leq C\left(\fint_{B_{\frac{3 r}{4}}} u^{l^{\prime}} \dx\right)^{\frac{1}{l^{\prime}}} +Ct
\no\\&{\leq C \,\underset{B_r}{\essinf}\,u+C\theta^{\frac{1}{p-1}}\left(\frac{r}{R}\right)^{\frac{p}{p-1}} \operatorname{Tail}\left(u^{-} ; x_0, R\right)}+Ct,\; \forall \, 0<l<\frac{d(p-1)}{d-p},
\end{align} 
where $C=C(l^\prime)$. We have neglected the dependence of $C$ on $l^\prime$ as $l^\prime=\la(d,p,s)$. For any $\varepsilon>0$, choosing
$$
t=\frac{1}{2}\theta^{\frac{1}{p-1}}\left(\frac{r}{R}\right)^{\frac{p}{p-1}} \operatorname{Tail}\left(u^{-} ; x_0, R\right)+\varepsilon,
$$
in \eqref{weakhar7} and letting $\varepsilon \rightarrow 0$, the result follows.  
\end{proof}

In the following remark, we note the relation between tail terms, average integral, $\esssup$ and $\essinf$ of $u_{\rho}$ and $u$. 
\begin{remark}\label{scaled to original} 
Recall $\theta$ from Remark \ref{scale}. For $\rho>0$, let $y_0= \rho x_0$. We observe that 
    \begin{align*}
      \theta  \{\Tail_{p-1,sp,p}\left(u_{\rho}^- ; x_0, R\right)\}^{p-1} & = R^{p} \rho^{d+sp}\rho^{p-sp}\,\int_{\rd\setminus B_R(x_0)} \frac{|u^-(\rho x)|^{p-1}}{|\rho x-y_0|^{d+sp}}\dx \\
      & = 2^{1-p} R^{p} \rho^{d+sp}\rho^{p-sp}\,\int_{\rd\setminus B_R(x_0)} \frac{||u(\rho x)|-u(\rho x)|^{p-1}}{|\rho x-y_0|^{d+sp}}\dx \\
        & = 2^{1-p} R^{p} \rho^{p} \,\int_{\rd\setminus B_{\rho R}(y_0)} \frac{||u(y)|-u(y)|^{p-1}}{|y-y_0|^{d+sp}} \dy\\
        &= R^{p} \rho^{p} \, \int_{\rd\setminus B_{\rho R}(y_0)} \frac{|u^-(y)|^{p-1}}{|y-y_0|^{d+sp}} \dy = \{\Tail_{p-1,sp,p}\left(u^- ; y_0, \rho R\right)\}^{p-1}.
    \end{align*}
   We also see that 
   \begin{align*}
    & \fint_{B_r(x_0)} u_{\rho}^l \dx = \frac{1}{\rho^d|B_r(x_0)|} \int_{B_{\rho r}(y_0)} u^l \dx = \frac{1}{\rho^d|B_r(y_0)|} \int_{B_{\rho r}(y_0)} u^l \dx = \fint_{B_{\rho r}(y_0)}  u^l \dx,  \\
    &  \underset{B_{r}(x_0)}{\esssup}\, u_{\rho} = \underset{B_{\rho r}(y_0)}{\esssup}\, u, \text{ and }  \underset{B_r(x_0)}{\essinf}\, u_{\rho} = \underset{B_{\rho r}(y_0)}{\essinf}\, u.
   \end{align*}
\end{remark}

\noi \textbf{Proof of Theorem \ref{harnack_intro} and Theorem \ref{weakhar_intro}:} For $\tilde{\Omega}$ as given in Remark \ref{scale}, we choose $\rho_0 \in (0,1)$ such that $\norm{V_{\rho_0}}_{\tilde{\Omega}}< \min\{\zeta, \hat{\zeta}\}$ and $R_0 = \rho_0R$. Then the proof follows combining Theorem \ref{harnack}, Theorem \ref{weakharnack}, and Remark \ref{scaled to original}. \qed

\vspace{0.2 cm}
\noindent \textbf{Acknowledgments:}	The research of N.B. is supported by the National Board for Higher Mathematics Postdoctoral Fellowship (0204/16(9)/2024/RD-II/6761). S.D. acknowledges the financial support provided by Indian Institute of Technology Kanpur. 

%\noindent \textbf{Author contributions statement:} N.B. and S.D. contributed equally to this work.

\noindent \textbf{Competing interests:} The authors have no competing interests to declare that are relevant to the content of this article.

\noindent\textbf{Data availability statement:} Data sharing does not apply to this article as no data sets were generated or analysed during the current study.

\bibliographystyle{abbrv}

\end{document}